\documentclass[a4paper]{amsart}
\usepackage[latin1]{inputenc}
\usepackage{amsfonts}
\usepackage{amsmath,latexsym,amssymb,amsfonts, amsthm}
\usepackage{esint}
\usepackage{cite}
\usepackage{graphicx}
\usepackage{amscd}
\usepackage{color}
\usepackage{bm}
\usepackage{enumerate}
\usepackage[dvips]{epsfig}
\usepackage{psfrag}
\usepackage{epsfig}
\usepackage[normalem]{ulem}

\usepackage[
  hmarginratio={1:1},
  vmarginratio={1:1},
  textwidth=16cm,
  textheight=22.3cm,
  heightrounded,
]{geometry}

\usepackage{graphicx,color}
\usepackage[colorlinks]{hyperref}
\hypersetup{linkcolor=blue,citecolor=blue,filecolor=black,urlcolor=blue}

\newtheorem{theorem}{Theorem}
\newtheorem{corollary}[theorem]{Corollary}
\newtheorem{proposition}[theorem]{Proposition}
\newtheorem{lemma}[theorem]{Lemma}
\theoremstyle{definition}
\newtheorem{remark}[theorem]{Remark}

\newtheorem{definition}[theorem]{Definition}

\numberwithin{theorem}{section}

\numberwithin{equation}{section}

\newcommand{\beq}{\begin{equation}}
\newcommand{\eeq}{\end{equation}}

\newcommand{\R}{\mathbb{R}}

\newcommand{\mf}[1]{\mathbf{#1}}
\newcommand{\bs}[1]{\boldsymbol{#1}}

\newcommand{\pa}{\partial}
\renewcommand{\S}{\mathbb{S}}

\newcommand{\ssubset}{\subset\joinrel\subset}

\newcommand{\dist}{{\rm dist}}
\newcommand{\supp}{{\rm supp}}

\newcommand{\weak}{\rightharpoonup}

\newcommand{\embed}{\hookrightarrow}

\newcommand{\eps}{\varepsilon}

\DeclareMathOperator{\loc}{loc}

\renewcommand{\epsilon}{\varepsilon}

\author[N. Soave and S. Terracini]{Nicola Soave and Susanna Terracini}\thanks{}
\address{Nicola Soave and Susanna Terracini\newline \indent
 Dipartimento di Matematica ``Giuseppe Peano'', Universit\`a di Torino, \newline \indent
Via Carlo Alberto 10,
10123 Torino, Italy}
\email{nicola.soave@unito.it, susanna.terracini@unito.it}

\title[On partial segregation models: H\"older estimates and properties of the limits]{On some singularly perturbed elliptic systems modeling partial segregation: uniform H\"older estimates and basic properties of the limits}
\thanks{N.S. is partially supported by the PRIN Project no. 2022R537CS ``Nodal Optimization, NOnlinear elliptic equations, NOnlocal geometric problems, with a focus on regularity ($NO^3$)", and S. T. is partially supported by the PRIN 2022 project 20227HX33Z -- \emph{Pattern formation in nonlinear phenomena},  funded by the European Union - Next Generation EU within the PRIN 2022 program (D.D. 104 - 02$/$02$/$2022 Ministero dell'Universit\`a e della Ricerca, Italy). Both authors are members of the INdAM--GNAMPA group.\\
Declarations of interest: none.}

\begin{document}
\begin{abstract}
We prove uniform H\"older estimates in a class of singularly perturbed competition-diffusion elliptic systems, with the particular feature that the interactions between the components occur three by three (ternary interactions).  These systems are associated to the minimization of Gross--Pitaevskii energies modeling ternary mixtures of ultracold gases and other multicomponent liquids and gases. We address the question of whether this regularity holds uniformly throughout the approximation process up to the limiting profiles, and we answer it in the affirmative. A very relevant feature of limiting profiles in this process is that they are only partially segregated, giving rise to new geometric pattern formation phenomena and to new questions on optimal regularity.
\end{abstract}

\maketitle

\section{introduction}

Various phenomena in physics, biology, chemistry, and pure mathematics can be described using reaction-diffusion systems with strongly competitive interaction. The intensity of this competition is parametrized by a singular perturbation coefficient $\beta$: as $\beta$ tends to infinity, the densities tend to segregate, giving rise to the formation of geometric  patterns in the domain. The study of the asymptotic behavior of solutions in the singular perturbation limit, as well as the properties of the limiting profile in terms of density regularity and geometry of the free boundary, has been the subject of extensive research (we postpone a detailed review of what is known in the literature to the following paragraphs). A common feature of almost all the published contributions is that they focus on strong binary interactions (species interact pairwise), which give rise to limiting models featuring full segregation, namely models in which different densities have disjoint positivity sets. In this paper we deal instead with a variational model with \emph{ternary interactions}  yielding \emph{partial segregation} in the limit. We emphasize that phase separation driven by multiple interactions has come to the fore in recent physical  literature on multicomponent liquids and gases and calls for a mathematical explanation (cf. e.g. \cite{LQZ24,Pet14} and references therein).

In the simplest possible setting, let $\Omega \subset \R^N$ be a bounded smooth domain, and let $(\psi_1, \psi_2, \psi_3)$ be a triplet of nonnegative Lipschitz continuous functions in $\overline{\Omega}$ satisfying the \emph{partial segregation condition}
\beq\label{PSC}
\psi_1\,\psi_2\,\psi_3 \equiv 0 \quad \text{in $\overline{\Omega}$};
\eeq
namely, $\psi_1$, $\psi_2$ and $\psi_3$ cannot be all positive at the same point $x_0$, but their positivity sets can overlap pairwise. Our goal is to analyze the asymptotic behavior as $\beta \to +\infty$ of solutions to the boundary value problem
\beq\label{PSP}
\begin{cases}
\displaystyle \Delta u_i = \beta u_i \prod_{\substack{j=1\\ j \neq i}}^3 u_j^2, \quad u_i>0 & \text{in }\Omega\\
u_i = \psi_i & \text{on }\partial \Omega,
\end{cases}
\eeq
for $i=1,2,3$. It is well established that, for every $\beta>0$ fixed, a classical solution $\mf{u}_\beta$ of the problem can be found by minimizing the energy with a ternary interaction term
\[
J_\beta(\mf{u}, \Omega):= \int_{\Omega} \big(\sum_{i=1}^3 |\nabla u_i|^2 + \beta \prod_{j=1}^3 u_j^2\big)\,dx
\]
in a Sobolev space of functions with fixed traces. It is also plain that, by minimality,
\[
J_\beta(\mf{u}_\beta,\Omega) \le J_\beta(\boldsymbol{\psi},\Omega) =  \int_\Omega \sum_{i=1}^3|\nabla \psi_i|^2\,dx,
\]
thanks to the partial segregation condition of $\boldsymbol{\psi}$. This implies that, up to a subsequence, $\mf{u}_\beta \rightharpoonup \tilde{\mf{u}}$ weakly in $H^1(\Omega,\R^3)$, and that the limit itself satisfies the partial segregation condition $\tilde u_1\,\tilde u_2\,\tilde u_3 \equiv 0$ a.e. in $\Omega$. We address the  following natural issues, closely related to each other:
\begin{itemize}
\item[($i$)] since $\mf{u}_\beta$ is smooth for each $\beta>0$ fixed, can we prove uniform-in-$\beta$ a priori bounds in some H\"older or higher-order Sobolev space?
\item[($ii$)] Can we describe the properties and the regularity of the limit $\tilde{\mf{u}}$?
\item[($iii$)] Can we describe the properties and the regularity of the \emph{free boundary} $\bigcup_1^3 \pa\{\tilde u_i>0\}$ and of the nodal set
 $\{x \in \Omega: \tilde u_i(x)=0 =\tilde u_j(x) \text{ for at least two indices } i \neq j\}$?
\end{itemize}
These questions will be answered in this paper and in its companion \cite{ST2024} (Part 2). More precisely, in this paper we establish uniform H\"older bounds for families of minimizers, with a small exponent $\bar\nu$, and derive the basic properties of the limiting profile. In \cite{ST2024}, we better investigate the optimal regularity of the limit, and discuss the geometry of the free boundary and its regularity.
\subsection{Main results} As said, our goal is to describe the asymptotic behavior, as $\beta \to +\infty$, of minimal solutions to \eqref{PSP}, with the boundary data $\boldsymbol{\psi}$ satisfying the partial segregation condition \eqref{PSC}. Actually, we consider a more general setting which allows to deal with family of solutions defined on different domains. Let $\Omega \subset \R^N$ be a domain (not necessarily bounded), and let $\{\mf{u}_\beta=(u_{1,\beta}, u_{2,\beta}, u_{3,\beta})\}_{\beta >1} \subset H^1_{\loc}(\Omega)$ be a family of weak solutions to
\begin{equation}\label{P beta}
\Delta u_{i} = \beta u_i \prod_{\substack{j \neq i}} u_{j}^2, \quad u_{i}>0 \quad \text{in $\Omega$, for $i=1,2,3$},
\end{equation}
under the following assumptions:
\begin{itemize}
\item[(h1)] $\{\mf{u}_\beta\}$ is uniformly bounded in $L^\infty(\Omega)$, namely $\|\mf{u}_\beta\|_{L^\infty(\Omega)} \le C$ for a positive constant $C>0$ independent of $\beta$.
\item[(h2)] $\mf{u}_\beta$ is a minimizer of \eqref{P beta} with respect to variations with compact support, in the sense that, for every $\Omega' \ssubset \Omega$
\[
J_\beta(\mf{u}_{\beta}, \Omega') \le J_\beta(\mf{u}_\beta + \bs{\varphi}, \Omega') \qquad \forall \bs{\varphi} \in H_0^1(\Omega', \R^3),
\]
where
\beq\label{def functional}
J_\beta(\mf{u}, \Omega') := \int_{\Omega'} \sum_{i=1}^3 |\nabla u_i|^2 \,dx+ \beta \int_{\Omega'} \prod_{j=1}^3 u_j^2\,dx.
\eeq
\end{itemize}

Our aim is to prove the validity of uniform $C^{0,\alpha}$ bounds, for some ``small" $\alpha \in (0,1)$, and to characterize the basic properties of the limit problem.

\begin{theorem}[Local uniform bounds in H\"older spaces]\label{thm: holder bounds}
Let $\mf{u}_\beta =(u_{1,\beta}, u_{2,\beta}, u_{3,\beta})$ be a solution of \eqref{P beta} at fixed $\beta>1$. Suppose that (h1) and (h2) hold. Then there exists $\bar \nu \in (0,2/3]$ depending on the dimension only, with $\bar \nu=2/3$ in dimension $N=2$, such that, for every compact set $K \ssubset \Omega$, and for every $\alpha \in (0,\bar\nu)$, we have that
\[
\|\mf{u}_\beta\|_{C^{0,\alpha}(K)} \le C,
\]
with $C=C(\alpha,N,K,\Omega,\sup_\beta\|\mf{u}_\beta\|_{L^\infty(\Omega)})>0$ independent of $\beta$. Moreover, as $\beta \to +\infty$
\begin{gather*}
\mf{u}_\beta \to \tilde{\mf{u}} \quad \text{in $H^1_{\loc}(\Omega)$ and in $C^{0,\alpha}_{\loc}(\Omega)$, for every $\alpha \in (0, \bar \nu)$, and} \\
\beta \int_{\omega} \prod_{j=1}^3 u_{j,\beta}^2\,dx \to 0 \quad \text{for every $\omega \ssubset \Omega$}
\end{gather*}
up to a subsequence.
\end{theorem}

It is immediate to check that such assumptions are satisfied if $\{\mf{u}_\beta\}$ is a family of solutions to \eqref{PSP} obtained by minimizing $J_\beta(\cdot\,,\Omega)$ in a space of functions with fixed Lipschitz traces. In such a case, we can prove that the uniform bounds extend \emph{up to the boundary}, and the limiting profile can be characterized as minimizer of a natural limit problem.

\begin{theorem}[Minimizers with fixed traces]\label{thm: min fixed traces}
Let $\Omega \subset \R^N$ be a bounded domain with $C^{1}$ boundary, let $\boldsymbol{\psi}=(\psi_1,\psi_2,\psi_3)$ be a nonnegative Lipschitz continuous function on $\overline{\Omega}$ satisfying the partial segregation condition \eqref{PSC}. For every $\beta>1$ fixed, let $\mf{u}_\beta =(u_{1,\beta}, u_{2,\beta}, u_{3,\beta})$ be a minimizer of $J_\beta(\cdot\,,\Omega)$ in
\[
H_{\boldsymbol{\psi}}:=\{ \mf{u} \in H^1(\Omega,\R^3): \ \mf{u}-\boldsymbol{\psi} \in H_0^1(\Omega,\R^3)\}.
\]
Then there exists a universal $\bar \nu \in (0,2/3]$ depending on the dimension only, with $\bar \nu=2/3$ in dimension $N=2$, such that
\beq\label{global bounds}
\|\mf{u}_\beta\|_{C^{0,\alpha}(\overline{\Omega})} \le C \qquad \forall \alpha \in (0, \bar \nu),
\eeq
with $C=C(\alpha,N,\Omega, \boldsymbol{\psi})>0$ independent of $\beta$. Moreover, as $\beta \to +\infty$
\begin{gather*}
\mf{u}_\beta \to \tilde{\mf{u}} \quad \text{in $H^1(\Omega)$ and in $C^{0,\alpha}(\overline{\Omega})$, for every $\alpha \in (0, \bar \nu)$, and} \\
\beta \int_{\Omega} \prod_{j=1}^3 u_{j,\beta}^2\,dx \to 0
\end{gather*}
up to a subsequence. Finally, the limit $\tilde{\mf{u}}$ is a minimizer for the problem
\beq\label{def c infty}
c_\infty:= \min\left\{ \sum_{i=1}^3 \int_\Omega |\nabla u_i|^2\,dx \left|\begin{array}{l}   \mf{u}-\boldsymbol{\psi} \in H_0^1(\Omega,\R^3) \\ u_1\,u_2\,u_3 \equiv 0 \ \text{in $\Omega$}\end{array}\right.\right\}.
\eeq
\end{theorem}

\begin{remark}
The threshold $\bar \nu$ is defined in terms of an optimal ``overlapping partition problem" on the sphere, and more precisely we have $\bar \nu= \alpha_{3,N}/3$, with $\alpha_{3,N}$ defined by \eqref{opp} ahead. In Section \ref{sec: ACF}, we will show that $\alpha_{3,N} \le 2$ in any dimension $N \ge 2$, with $\alpha_{3,2}=2$, so that Theorem \ref{thm: holder bounds} provides uniform H\"older bounds for exponents smaller than $\bar \nu \le 2/3$. This is a suboptimal result. The determination of the optimal exponent will be the subject of analysis in our next work \cite{ST2024}, but we can anticipate here that we will prove it to be $3/4$, see \cite[Theorems 1.6 and 1.7]{ST2024}. We point out that, in the derivation of the optimal uniform bounds, a key point will be  Theorem \ref{thm: holder bounds} here, complemented with an \emph{improvement-of-regularity} argument.
\end{remark}

\begin{remark}\label{rem:scale-invariant}
Assumption (h1) can be dispensed with altogether, provided that the
conclusion of Theorem \ref{thm: holder bounds} is stated in
scale-invariant form. Precisely, let $\bar\nu$ be as in
Theorem \ref{thm: holder bounds}. For every $\alpha \in (0,\bar\nu)$ and
every pair of open sets $\Omega' \subset\subset \Omega'' \subset\subset
\Omega$ there exists a constant $C = C(N,\alpha,\Omega',\Omega'') > 0$
such that
\begin{equation}\label{eq:scale-invariant}
  \|\mathbf{u}\|_{C^{0,\alpha}(\overline{\Omega'})}
  \;\le\; C \, \|\mathbf{u}\|_{L^{\infty}(\Omega'')}
\end{equation}
for every $\beta > 0$ and every solution $\mathbf{u}$ of
\eqref{P beta} in $\Omega$ satisfying (h2). In particular, $C$ depends
neither on $\beta$ nor on the solution $\mathbf{u}$, and no a priori
bound on $\|\mathbf{u}\|_{L^\infty}$ is assumed: the $L^\infty$ norm of
the solution itself appears on the right hand side of
\eqref{eq:scale-invariant}.

Estimate \eqref{eq:scale-invariant} contains Theorem
\ref{thm: holder bounds}: if a family $\{\mathbf{u}_\beta\}_{\beta>1}$
satisfies (h1), the right hand side is bounded uniformly in $\beta$. It
is however strictly stronger, in that it applies to families whose
$L^\infty$ norms degenerate or blow up as $\beta \to +\infty$, at the
only price of a bound which degenerates at the same rate.

A similar variant can also be proved for Theorem \ref{thm: min fixed traces}. Let $\Omega \subset \mathbb{R}^N$ be a bounded domain with $C^1$
boundary, and let $\bar\nu$ be as in Theorem
\ref{thm: min fixed traces}. For every $\alpha \in (0,\bar\nu)$ there
exists a constant $C = C(N,\alpha,\Omega) > 0$ such that
\begin{equation}\label{eq:scale-invariant-boundary}
  \|\mathbf{u}\|_{C^{0,\alpha}(\overline{\Omega})}
  \;\le\; C \, \|\boldsymbol{\psi}\|_{C^{0,1}(\overline{\Omega})}
\end{equation}
for every $\beta > 0$, every nonnegative Lipschitz continuous
$\boldsymbol{\psi}$ on $\overline{\Omega}$ with
$\psi_1 \psi_2 \psi_3 \equiv 0$ in $\overline{\Omega}$, and every
minimizer $\mathbf{u}$ of $J_\beta(\,\cdot\,,\Omega)$ in
$H_{\boldsymbol{\psi}}$. In particular the constant is independent of the
boundary datum, and \eqref{eq:scale-invariant-boundary} makes the
dependence of estimate (1.5) on $\boldsymbol{\psi}$ explicit and linear
in the Lipschitz norm. The proof of this estimate is analogous to that of estimate \eqref{eq:scale-invariant}, so we omit the details.
\end{remark}

The proof of Theorem \ref{thm: holder bounds} is based on a contradiction argument and on a delicate blow-up analysis.
In particular, one has to carefully deal with possible blow-up limits exhibiting partial segregation, but also with blow-up profiles where two components are fully segregated (say $u_1 u_2 \equiv 0$), and the third one ``does not see" the others. To deal with the latter scenario, we can take advantage on previous contributions regarding full segregation models (conveniently adapted), while the study of the former case rests on two pillars, which we believe are findings of independent interest and liable to further applications. The first is a proper multiple-phases Alt--Caffarelli--Friedman type monotonicity formula (Theorem \ref{thm: ACF}) tailored for the case of partially segregated densities, and its perturbed version. The second backbone is a Liouville-type theorem for entire solutions to \eqref{P beta} clarifying the minimal growth of its non-trivial solutions. Here is a statement, which follows immediately from Theorem \ref{thm: liou sub}:
\begin{corollary}[Liouville-type Theorem]
Let $k \ge 3$ and $N \ge 2$ be positive integers, and, for $i=1,\dots,k$, let $u_i \in H^1_{\loc}(\R^N) \cap C(\R^N)$ be nonnegative functions satisfying the following system of PDEs
\[
\Delta u_i = M u_i \prod_{\substack{j=1 \\ j \neq i}}^k u_j^2, \quad u_i \,\ge  \,0 \quad \;{\rm in}\; \R^N,
\]
for some $M>0$. There exists $\bar\nu\in(0,1)$, depending only on the dimension $N$ and the number of components $k$, such that, if
\[
0\le u_i(x) \le C(1+|x|^{\alpha}) \quad \text{for every $x \in \R^N$},
\]
with $\alpha <\bar\nu$, then at least one component $u_i$ vanishes identically and all others are constant.
\end{corollary}

In other words, three phases can coexist, in a ternary interaction regime, only if they feature a spatial growth greater than or equal to $|x|^{\bar\nu}$. We will prove in  the forthcoming \cite{ST2024} that the optimal exponent is exactly $\bar\nu=3/4$, hence smaller than one, while we know that in the regime of two-by-two interactions the minimal growth exponent is $3/2$ (cf. \cite[Corollary 1.9]{ST15}).

\medskip

Once Theorem \ref{thm: holder bounds} is established, it is not difficult to derive the basic properties of the limiting profiles. We first introduce a class of triplets $\mf{v}:\Omega \to \R^3$ satisfying the partial segregation condition.

\begin{definition}\label{def: L}
Let $\Omega \subset \R^N$ be a domain. We denote by $\mathcal{L}(\Omega)$ the set of nonnegative functions $\mf{v} \in H^1(\Omega,\R^3) \cap C(\Omega,\R^3)$, $\mf{v} \not \equiv 0$, satisfying the following conditions: there exist sequences $M_n \to +\infty$ and $\{\mf{v}_n\} \subset H^1(\Omega,\R^3)$, with $v_{i,n}>0$ in $\Omega$, such that
\begin{itemize}
\item[($i$)] $\mf{v}_n$ is a minimizer of $J_{M_n}$ with respect to variations with fixed trace in $\Omega$, namely
\[
J_{M_n}(\mf{v}_n, \Omega) \le J_{M_n}(\mf{v}_n + \bs{\varphi}, \Omega) \qquad \forall \bs{\varphi} \in H_0^1(\Omega,\R^3).
\]
\item[($ii$)] $\mf{v}_n \to \mf{v}$ strongly in $H^1(\Omega)$, and in $C^{0}_{\loc}(\Omega)$.
\item[($iii$)] As $n \to \infty$
\[
\int_\Omega M_n \prod_{j=1}^3 v_{j,n}^2\,dx \to 0.
\]
\end{itemize}
We denote by $\mathcal{L}_{\loc}(\Omega)$ the set of functions $\mf{v}$ such that $\mf{v} \in \mathcal{L}(\omega)$ for every $\omega \ssubset \Omega$.
\end{definition}

\begin{remark}\label{rem: limit class}
The second part of Theorem \ref{thm: holder bounds} can now be summarized by saying that, up to a subsequence $\mf{u}_\beta \to \tilde{\mf{u}} \in \mathcal{L}_{\loc}(\Omega) \cup \{0\}$ as $\beta \to +\infty$.
\end{remark}

\begin{theorem}[H\"older continuity and Local Poho\v{z}aev identity for limiting profiles]\label{thm: bas prop}
Let $\tilde{\mf{u}} \in \mathcal{L}_{\loc}(\Omega)$. Then $\tilde{\mf{u}} \in C^{0,\alpha}_{\loc}(\Omega) \cap H^1_{\loc}(\Omega)$, for every $\alpha \in (0, \bar \nu)$,
\beq\label{cond base 1}
\Delta \tilde u_i=0 \quad \text{in $\{\tilde u_i>0\}$} \quad \text{and} \quad \tilde u_1\,\tilde u_2\,\tilde u_3 \equiv 0 \quad \text{in $\Omega$},
\eeq
and the following \emph{domain variation formula} (or \emph{local Pohozaev identity}) holds:
\beq\label{loc Poh}
\int_{S_r(x_0)} \sum_i |\nabla \tilde u_i|^2\,d\sigma = \frac{N-2}{r} \int_{B_r(x_0)} \sum_i |\nabla \tilde u_i|^2\,dx + 2 \int_{S_r(x_0)} \sum_i (\pa_\nu \tilde u_i)^2\,d\sigma,
\eeq
whenever $B_r(x_0) \subset \Omega$.
\end{theorem}

Theorem \ref{thm: bas prop} collects the \emph{main extremality conditions} satisfied by the limiting profiles of system \eqref{PSP}. They will be used in \cite{ST2024} in order to study the free-boundary and the optimal regularity.

\begin{remark}

The assumption that $\pa \Omega$ is of class $C^1$ in Theorem \ref{thm: min fixed traces} cannot be relaxed to mere
Lipschitz regularity. On the one hand, the method breaks down: in the blow-up
procedure of Subsection \ref{sub: boundary} we need the scaled domains
$\Omega_n$ to converge either to $\R^N$ or to a half-space, whereas the boundary
blow-ups of a Lipschitz domain are cones, which moreover need not be unique. On
the other hand, and more substantially, the statement itself becomes false.

Indeed, let $N=2$, let $\Theta \in \left(\frac32 \pi, 2\pi\right)$, and let $\Omega := \left\{ (r,\theta) : \ 0 < r < 1, \ 0 < \theta < \Theta \right\}$ be a circular sector of $B_1$ of opening $\Theta$, which is a Lipschitz domain.
Let $c>0$, and let $\bs{\psi}$ be any nonnegative Lipschitz continuous function
on $\overline{\Omega}$ whose trace on $\pa \Omega$ is
\[
  \left( \sin\left(\frac{\pi \theta}{\Theta}\right), c, 0 \right);
\]
such a function exists, since the above trace is Lipschitz continuous on
$\pa \Omega$, and it satisfies the partial segregation condition \eqref{PSC}.
Since $\psi_3 \equiv 0$, the third component of any minimizer vanishes
identically, so that the interaction term does as well, and the remaining two
components are the harmonic extensions of their own traces. Therefore, for every
$\beta > 0$,
\[
  \mf{u}_\beta(r,\theta) = \left( r^{\pi/\Theta} \sin\left(\frac{\pi \theta}{\Theta}\right),
   c,  0 \right)
\]
is a minimizer of $J_\beta(\cdot\,,\Omega)$ in $H_{\bs{\psi}}$. Its first
component is $\pi/\Theta$-homogeneous at the vertex, and hence it does not belong
to $C^{0,\alpha}(\overline{\Omega})$ for any $\alpha > \pi/\Theta$. Since
$\Theta > \frac32\pi$, we have $\pi/\Theta < 2/3 = \bar\nu$; consequently estimate
\eqref{global bounds} fails for every $\alpha \in (\pi/\Theta, 2/3)$, even though
such exponents are admissible in Theorem \ref{thm: min fixed traces}.

\end{remark}

\begin{remark}
One may wonder whether Theorem \ref{thm: bas prop} - and hence also the subsequent analysis in \cite{ST2024} - holds for any minimizer for $c_\infty$ (defined in \eqref{def c infty}), or not. Notice that it is not clear that any such minimizer can be approximated by solutions to \eqref{PSP}, and hence may not stay in $\mathcal{L}_{\loc}(\Omega)$. Using a simple and standard penalization argument, one can easily prove that \emph{all energy minimizers} among configurations with partial segregation as in \eqref{def c infty} enjoy the same properties as stated in the theorem. The penalization procedure does not present substantial difficulties, but it does cause an additional term to appear in equation \eqref{PSP}, leading to some purely technical, sometimes cumbersome, complications. For this reason, we have decided to omit the details. The interested reader can refer to \cite{ST19} for a similar situation, where all steps are made explicit.

We point out that, despite the convexity of the functional to be minimized, it is not difficult to see that minimizers of \eqref{def c infty} are generally not unique, owing to the singular constraint. Nonuniqueness can already be observed in dimension $N=1$, by considering the problem on $[0,1]$ with boundary data
$$
u_1(0)=a=u_1(1), \quad u_2(0)=1=u_3(1), \quad u_3(0)=0=u_2(1),
$$
where $a>0$. Each component is harmonic, hence linear, wherever it is positive. Explicit computations yield the threshold $\bar a:= \sqrt{2}-1>0$ such that, for $a>\bar a$, the problem admits a unique minimizer, characterized by the property $u_1\equiv a$, whereas for $a<\bar a$, the component $u_1$ vanishes on an interval. At the threshold $a=\bar a$, continuity yields two distinct minimizers.
\end{remark}

The present contribution and the forthcoming \cite{ST2024} are among the first attempts to study partial segregation phenomena in a systematic way. As far as we know, partial segregation models have previously been studied only in \cite{BoBuFo, CaRo}. In \cite{CaRo}, the authors focused on uniform-in-$\beta$ H\"older bounds (for some small exponent $\alpha \in (0,1)$) for uniformly bounded solutions of rather general systems of type
\[
L_i u_i = - \beta A_i(x) F(u_1,\dots,u_k), \quad F(u_1,\dots,u_k)= \prod_{j=1}^k u_j^{\sigma_j},
\]
where $L_i$ is an elliptic operator with possibly variable coefficient, and $A_i$ are smooth and positive. The case
\beq\label{symPS}
\Delta u_i = \beta \prod_{j=1}^k u_j
\eeq
is included in their analysis. Notice that the function $F(u_1,\dots,u_k)$ appearing on the right hand side is the same for all the components;
this is a key feature in their analysis (and marks a significant difference with our case). In \cite{BoBuFo}, the authors conducted a more thorough analysis of the properties of the limiting profiles for system \eqref{symPS}. Contrary to our case, such profiles appear to be Lipschitz continuous, while we show (cf. \cite{ST2024}, Part 2) the maximal exponent of H\"older continuity to be strictly less than one, thus preventing boundedness of the gradients.

\begin{remark}
The system \eqref{PSP} describes a triple of components that interact simultaneously through a three-way interaction from a variational point of view; specifically, through the interaction term $\beta u_1^2 u_2^2 u_3^2$ appearing in the energy functional, the simultaneous presence of all three components at the same point is penalized. This leads to partial segregation in the limit as the singular perturbation parameter $\beta$ tends to $+\infty$. It is quite evident that the techniques introduced can also be applied to systems with $k \ge 3$ components, interacting three by three. This introduces some complications of a mainly technical nature, which we preferred to avoid at this initial stage. A significant difference lies in the fact that, in the strong competition limit, one ends up studying a local minimum of the energy under the condition of partial segregation $u_i u_j u_k \equiv 0$ in $\Omega$, for every triple $(i,j,k)$ with $i \neq j$, $j \neq k$, $k \neq i$.

A further generalization, which differs from the previous one in a more substantial way, involves $k \ge 3$ components interacting all together, i.e., $k$ by $k$, via an interaction term of the form
\[
\beta \int u_1^2 \cdots u_k^2 \,dx.
\]
The techniques developed in this work are still useful for the study of this problem; however, the validity of the H\"older bounds for $k$ components requires knowledge of the validity of the H\"older bounds for the $(k-1)$-component problem. Moreover, it is not difficult to show that the optimal regularity exponent of the minimizers changes. For this reason, we have deferred the study of this more general case to future work.
\end{remark}

\subsection{Two-by-two interactions} To give an idea of the scope of our results and the new difficulties to be overcome, let us now make some comparisons with the case, already extensively dealt with in the literature, of binary interactions and the resulting total segregation limit configurations. This serves also as motivation for our study. When condition \eqref{PSC} is replaced by the \emph{(full) segregation condition}
\[
\phi_i\,\phi_j \equiv 0 \quad \text{in $\Omega$ for every $i \neq j$}
\]
(namely the positivity sets of $\phi_i$ and $\phi_j$ are disjoint), problem
\beq\label{FSP}
\begin{cases}
\displaystyle \Delta u_i = \beta u_i \sum_{j \neq i} u_j^2, \quad u_i>0  & \text{in }\Omega\\
u_i = \phi_i & \text{on }\partial \Omega
\end{cases}
\eeq
and its generalizations were studied in \cite{CL08, NTTV10, Wa}, see also \cite{CafLin06, CTV03, CTVind, DWZjfa, STTZ16, TT12}. Also in this case one can find a solution $\mf{v}_\beta$ by minimizing the associated functional, and prove that a family of (positive) minimizers weakly converges in $H^1$ to a fully segregated nonnegative function $\bar{\mf{u}}$, namely $\bar u_i \, \bar u_j \equiv 0$ a.e. in $\Omega$, for every $i \neq j$. At this point, one can consider the same questions ($i$)-($iii$) raised before. The picture is now well understood, and optimal results are available. It is not difficult to show that $\bar{\mf{u}}$ is a \emph{harmonic map into a singular space}, in the sense that it is a minimizer for
\[
\min\left\{ \int_{\Omega} \sum_{i=1}^3 |\nabla u_i|^2\,dx \left| \begin{array}{l} \mf{u} \in H^1(\Omega,\R^3), \ u_i = \phi_i \quad \text{ on $\pa \Omega$} \\ u_i\,u_j \equiv 0 \quad \text{a.e. in $\Omega$, for every $i \neq j$}\end{array}\right.\right\}.
\]
Notice that the target space is
\[
\left\{x \in \R^3: \ x_i \,x_j = 0 \quad \text{for every $i \neq j$, and $x_i \ge 0$ for every $i$}\right\},
\]
which is a singular space with non-positive curvature in the sense of Alexandrov, see \cite{GrSc}. Here we see another fundamental difference with the case under investigation in this paper; in our case the target, not being uniquely geodesic, cannot have non-positive curvature. This aspect will be further explored in the forthcoming \cite{ST2024}. Regarding uniform a priori bounds, it is proved that $\{\mf{v}_\beta: \beta>1\}$ is bounded in $C^{0,\alpha}_{\loc}(\Omega)$, for every $\alpha \in (0,1)$. Despite the fact that each function $\mf{v}_\beta$ is smooth, this is almost optimal, since the optimal regularity for the limit $\bar{\mf{u}}$ turns out to be the Lipschitz continuity. The uniform boundedness in $C^{0,\alpha}_{\loc}$ entails a number of consequences: it is clear that, up to a subsequence, $\mf{v}_\beta \to \bar{\mf{u}}$ in $C^{0,\alpha}_{\loc}(\Omega)$, for every $\alpha$. Moreover, it also allows to prove the strong $H^1_{\loc}$ convergence, and to derive the basic properties of the limit $\bar{\mf{u}}$: it satisfies
\beq\label{prop lim full}
\Delta \bar u_i = 0 \quad \text{in }\{u_i>0\}, \quad \text{and} \quad u_i \, u_j \equiv 0 \quad \text{in $\Omega$, for every $i \neq j$},
\eeq
and, in view of the variational structure of problem \eqref{FSP}, the \emph{domain variation formula} (or \emph{local Pohozaev identity}) \eqref{loc Poh} holds
whenever $B_r(x_0) \subset \Omega$. Notice that the domain variation formula is the same as the one arising in the partial segregation model, but the fact that it is coupled with the condition of total or partial segregation results in significant differences in the way it can be exploited. In both cases, it is worth to remark that this identity keeps track of the interaction between the different components in the singular limit, an interaction that is lost at the pointwise level (see \eqref{cond base 1} or \eqref{prop lim full}), and it has been the key ingredient to study the free boundary regularity in full segregation models: the free boundary $\Gamma:= \bigcup \pa\{\bar u_i>0\}$ and the \emph{nodal set} $\{\bar{\mf{u}}=0\}$ (set of points where all the components vanish) coincide; $\Gamma$ has Hausdorff dimension $N-1$ and, more precisely, it splits into a regular part, which is a collection of $C^{1,\alpha}$-hypersurfaces, and a singular set of dimension at most $N-2$ (optimal); the regular part $\mathcal{R}$ is characterized by the fact that for any point $x_0 \in \mathcal{R}$ there exists a radius $\rho>0$ such that exactly two components, say $\bar u_1$ and $\bar u_2$, do not vanish identically in $B_\rho(x_0)$, and $\bar u_1-\bar u_2$ is harmonic in $B_\rho(x_0)$. Finally, in dimension $N=2$ the singular set is discrete, and the regular part consists of a locally finite collection of curves meeting with equal angles at singular points.

We also refer to \cite{Alp, SZ15, SZ16} for further results such as uniform \emph{Lipschitz} bounds, sharp pointwise decay estimates for solutions to \eqref{FSP}, and fine properties of the singular set in higher dimension.

The aforementioned  results admit an extension into more general contexts: they also hold for non-minimal, sign changing, solutions of systems with reaction terms and \emph{competing groups} of components. We refer the interested reader to \cite{STTZ16} for more details. Moreover, the same issues were considered for system \eqref{FSP} with the Laplacian replaced by the fractional Laplacian \cite{TVZ16, TVZ14, ToZi20}, for systems with \emph{long range interaction} \cite{CPQ, STTZ18, STZ23}, for asymmetric systems with different operators for each component \cite{ST23}, and for systems with non-variational interaction \cite{CTV05, CKL, TT12}, such as
\[
\begin{cases}
\displaystyle \Delta u_i = \beta u_i \sum_{j \neq i} u_j, \quad u_i>0 & \text{in }\Omega\\
u_i = \phi_i & \text{on }\partial \Omega.
\end{cases}
\]
In general, the results obtained depend on the particular problem considered, but a common feature is that all these models describe phenomena of full segregation, namely, at each point at most one component of the limiting profile can be non-zero.

\subsection*{Structure of the paper} In Section \ref{sec: ACF}, we prove an Alt--Caffarelli--Friedman monotonicity formula tailored for the case of partially segregated densities, and its perturbed version. These monotonicity formulas are used in Section \ref{sec: liou} to prove some Liouville-type theorems. The core of the proof of Theorem \ref{thm: holder bounds} is the content of Section \ref{sec: hold bounds}; in Subsection \ref{sub: scale-invariant}, we also present the sketch of the proof for the scale-invariant version, see Remark \ref{rem:scale-invariant}. Sections \ref{sec: corol fixed} and \ref{sec: prop} are devoted to the proof of Theorem \ref{thm: min fixed traces} and Theorem \ref{thm: bas prop}, respectively.

\subsection*{Basic notation} We conclude this introduction with some basic notation which will be used throughout the rest of the paper.
\begin{itemize}
\item We denote by $B_r(x_0)$ the ball of center $x_0$ and radius $r$, and by $S_r(x_0)$ its boundary. When $x_0=0$, we simply write $B_r=B_r(0)$ and $S_r=S_r(0)$.
\item For $u \in H^1(\S^{N-1})$, we consider the optimal value
\[
\lambda(u) := \inf\left\{ \frac{\int_{\S^{N-1}}  |\nabla \varphi|^2 \,d\sigma   }{\int_{\S^{N-1}} \varphi^2 \,d\sigma}
\left| \begin{array}{l} \varphi \in \,H^1(\S^{N-1}) \setminus \{0\} \ \text{and}\\ \mathcal{H}^{N-1}(\{\varphi \neq 0\} \cap \{u=0\}) = 0  \end{array}\right.
\right\}.
\]
where $d\sigma = d\sigma_x$ and $\mathcal{H}^{N-1}$ stands for the usual $(N-1)$-dimensional Hausdorff measure. Notice that, if $u$ is also continuous, then $\lambda(u)$ is the first eigenvalue of the Laplace--Beltrami operator with homogeneous Dirichlet boundary condition on the open set $\{\xi \in \S^{N-1}: u(\xi) >0\}$.
\item For $u \in H^1(S_r(x_0))$, we set $u_{x_0,r}(\xi) = u(x_0+r \xi) \in   H^1(S_1) \simeq H^1(\S^{N-1})$.
\item We define $\gamma: [0,+\infty) \to [0,+\infty)$
\[
\gamma(t) := \sqrt{\left(\frac{N-2}{2}\right)^2 + t\,} - \frac{N-2}{2}.
\]
\item For $N \ge 3$ and $\delta>0$, we define $\phi_\delta: [0,+\infty) \to (0,+\infty)$ and $\Phi_\delta: \R^N \to (0,+\infty)$ by
\begin{equation}\label{reg_fun}
\phi_\delta(r) =  \begin{cases}
\frac{N}2 \delta^{2-N} + \frac{2-N}2\delta^{-N} r^2 & \text{if }0 \le r \le \delta \\
r^{2-N} & \text{if $r>\delta$},
\end{cases} \qquad \Phi_\delta(x) = \phi_\delta(|x|).
\end{equation}
$\Phi_\delta$ is a $C^1$ positive superharmonic function in $\R^N$.
\item We often use the vector notation $\mf{u}:=(u_1,\dots,u_k)$ for functions $\R^N \to \R^k$. Moreover, we denote by $\hat{\mf{u}}_i: \R^N \to \R^{k-1}$ the vector obtained from $\mf{u}$ by erasing the $i$-th component.
\item We say that a function $v$ is \emph{globally $\alpha$-H\"older continuous in an open set $\Omega$}, for some $\alpha \in (0,1)$, if the H\"older seminorm
\[
[v]_{C^{0,\alpha}(\Omega)}:= \sup_{\substack{x \neq y \\ x, y \in \Omega}} \frac{|v(x)-v(y)|}{|x-y|^\alpha}
\]
is finite. No limitation on the $L^\infty$ norm of $v$ is required.
\end{itemize}

\section{An Alt--Caffarelli--Friedman monotonicity formula for partially segregated phases}\label{sec: ACF}

The purpose of this section is to establish an ACF-type monotonicity formula for vector valued nonnegative subharmonic functions satisfying the partial segregation condition, and some variants. This is in the spirit of the classical ACF monotonicity formula proved in \cite{AlCaFr}, which concern totally segregated functions.

For $N \ge 2$, we define
\[
I(u,x_0,r) := \int_{B_r(x_0)} \frac{|\nabla u|^2}{|x-x_0|^{N-2}}\,dx,
\]
and
\[
J_\nu(u_1,\dots,u_k,x_0,r) := \frac{1}{r^{2\nu}} \prod_{j=1}^k I(u_j,x_0,r)
\]

\begin{lemma}\label{lem: ACF pre}
Let $u \in H^1_{\loc}(B_R)$ be nonnegative, and such that $\Delta u \ge 0$ in $B_R$. Then, for every $x_0 \in B_R$ and for almost every $r >0$ such that $B_r(x_0) \ssubset B_R$, we have
\[
\int_{S_r(x_0)} \frac{|\nabla u|^2}{|x-x_0|^{N-2}}\,d\sigma \ge \frac{2 \gamma( \lambda(u_{x_0,r}))}{r}  \int_{B_r(x_0)} \frac{|\nabla u|^2}{|x-x_0|^{N-2}}\,dx.
\]
where $u_{x_0,r}=u(x_0+r\xi)$.
\end{lemma}

The proof of the lemma is rather standard. We refer to \cite[Chapter 2]{PSU12} for a proof under an extra continuity assumption, and to \cite[Proof of Theorem 1.3]{Vel} and \cite[Lemma 2.2]{ST23} for the general case (see also \cite{AlCaFr}).

\begin{remark}\label{rmk acf pre}
Notice that the lemma is valid also in the case when $\lambda(u_{x_0,r}) = 0$, which takes place when $u >0$ everywhere on $S_r(x_0)$. Indeed, in such a situation the left hand side of the inequality is nonnegative, and the right-hand side is $0$, being
\[
\lambda(u_{x_0,r}) = \lambda_1(\S^{N-1}) = 0.
\]
\end{remark}

We study now the following optimization problem with partial segregation:
\begin{equation}\label{opp}
\alpha_{k,N}:=\inf\left\{ \sum_{j=1}^k \gamma(\lambda(u_j)) \left| \begin{array}{l} u_j \in H^1(\S^{N-1})  \\
\int_{\S^{N-1}} u_1^2 \cdots u_k^2\, d\sigma = 0.
\end{array}\right.\right\},
\end{equation}
with the convention that $\lambda(u) = +\infty$ if $u \equiv 0$ on $\S^{N-1}$ (this gives some continuity to $\lambda(\cdot)$ since, if $\mathcal{H}^{N-1}(\{u_n>0\}) \to 0$, then
$\lambda(u_n) \to +\infty$ by the Sobolev inequality). The infimum is greater than or equal to $0$, since we are minimizing the sum of nonnegative quantities. We also recall that $\alpha_{2,N}=2$ (this is the content of the Friedland--Hayman inequality \cite{FrHa}\footnote{The Friedland--Hayman inequality is usually stated in a slightly different form, involving partitions of the sphere into disjoint open sets. However, in light of the condition $\int_{\S^{N-1}} u_1^2 u_2^2\,d\sigma = 0$, it is not difficult to check that the inequality is equivalent to the fact that $\alpha_{2,N}=2$.}), and the optimal value is reached if and only if $u_1$ and $u_2$ are $1$-homogeneous functions supported on disjoint hemispheres (see \cite[Chapter 2]{PSU12} and references therein for more details).

\begin{lemma}\label{lem: on ov}
For every pair of integers $k\ge 3$ and $N\ge 2$, we have $0<\alpha_{k,N}\le 2$. Moreover,
$\alpha_{k,2}=2$ for every $k \ge 3$.
\end{lemma}

Determining the explicit value of $\alpha_{k,N}$ in every dimension remains an open problem.

\begin{proof}
We first prove that $\alpha_{k,N}>0$. Suppose by contradiction that $\alpha_{k,N} = 0$. Since $\gamma(t) > 0$ for $t > 0$, $\gamma(0)=0$, and $\gamma$ is continuous, there exists a minimizing sequence of $L^2$-normalized functions $\{(u_{1,n}, \dots, u_{k,n})\} \subset H^1(\S^{N-1}, \R^k)$ such that as $n \to \infty$
\[
\int_{\S^{N-1}} |\nabla u_{j,n}|^2\, d \sigma \to 0, \qquad \int_{\S^{N-1}} u_{j,n}^2\, d\sigma =1, \qquad \forall j=1,\dots,k
\]
and for every $n$
\[
\int_{\S^{N-1}} u_{1,n}^2 \dots u_{k,n}^2\, d \sigma = 0.
\]
We deduce that up to a subsequence $(u_{1,n},\dots,u_{k,n})$ tends to $(u_1,\dots,u_k)$ weakly in $H^1(\S^{N-1},\R^k)$, strongly in $L^2(\S^{N-1},\R^k)$, and almost everywhere on $\S^{N-1}$. Therefore
\[
\int_{\S^{N-1}} |\nabla u_{j}|^2\, d \sigma = 0, \qquad
\int_{\S^{N-1}} u_j^2 \,d\sigma= 1, \qquad \forall j=1,\dots,k,
\]
whence it follows that all the components $u_i$ are equal to the same constant $c_N$ on $\S^{N-1}$; but on the other hand, by Fatou's lemma, we also have
\[
\int_{\S^{N-1}} u_{1}^2 \dots u_{k}^2\, d \sigma = 0,
\]
which gives a contradiction.

Let us prove now that $\alpha_{k,N} \le 2$. To this purpose, we choose $(u_1,\dots,u_k)$ as follows:
\[
u_j(x) = \begin{cases} x_1^+ & \text{if $j=1$} \\
x_1^- & \text{if $j=2$} \\ c & \text{if $j \neq 1,2$}, \end{cases}
\]
where $c>0$, as admissible competitor for $\alpha_{k,N}$. Recalling that $\lambda(c) = 0$, we obtain
\[
\alpha_{k,N} \le \sum_{j=1}^k \gamma(\lambda(u_j)) = \gamma(\lambda(x_1^+)) + \gamma(\lambda(x_1^-)) = 2\gamma(N-1) = 2,
\]
which is the desired estimate.

It remains to prove that $\alpha_{k,2}=2$; in view of what precedes, it is
sufficient to show that $\alpha_{k,2} \ge 2$. Let $(u_1,\dots,u_k)$ be admissible
for \eqref{opp} with $N=2$, and let $\omega_j := \{u_j \neq 0\}$. Since
$H^1(\S^1) \embed C(\S^1)$, each $\omega_j$ is an open subset of $\S^1$, and the
constraint forces
$\bigcap_{j=1}^k \omega_j = \emptyset$, an open set of vanishing measure being
empty. Moreover $\gamma(t) = \sqrt{t}$ for $N=2$, and the first Dirichlet
eigenvalue of an open arc $I \subset \S^1$ is $(\pi/|I|)^2$; therefore
\[
\gamma(\lambda(u_j)) = \frac{\pi}{\ell_j},
\qquad \text{where} \qquad
\ell_j := \sup\big\{ |I| : \ \text{$I$ is a connected component of $\omega_j$} \big\},
\]
with the conventions $\ell_j = +\infty$ if $\omega_j = \S^1$ (so that
$\lambda(u_j)=0$) and $\ell_j = 0$ if $u_j \equiv 0$.

Let $P := \{ j : \omega_j \neq \S^1\}$, and let $p$ denote its cardinality. If
$p=0$, then the $k$-tuple is not
admissible. If $p=1$, say $P = \{j_0\}$, then
$\bigcap_j \omega_j = \omega_{j_0} = \emptyset$, so that $u_{j_0} \equiv 0$ and
$\sum_j \gamma(\lambda(u_j)) = +\infty$. We may therefore suppose that $p \ge 2$.
Let $\eps>0$, and for every $j \in P$ let $I_j$ be a connected component of
$\omega_j$ with $|I_j| > \ell_j - \eps$. Since
$\bigcap_{j \in P} I_j \subset \bigcap_{j=1}^k \omega_j = \emptyset$, the closed
arcs $\S^1 \setminus I_j$, with $j \in P$, cover $\S^1$; consequently
\[
\sum_{j \in P} \big( 2\pi - |I_j| \big)  \ge 2\pi,
\qquad \text{whence} \qquad
\sum_{j \in P} \ell_j  \le 2\pi(p-1),
\]
the latter inequality being obtained by letting $\eps \to 0^+$. The constraint has
thus become linear in the lengths $\ell_j$, and the Cauchy-Schwarz inequality
gives
\[
\sum_{j=1}^k \gamma(\lambda(u_j))  = \sum_{j \in P} \frac{\pi}{\ell_j}
 \ge \frac{\pi \, p^2}{\sum_{j \in P} \ell_j}
 \ge \frac{p^2}{2(p-1)}  =: f(p).
\]
Since $f'(p) \ge 0$ for $p \ge 2$, the function $f$ is
non-decreasing on $[2,+\infty)$, and hence
$\sum_{j=1}^k \gamma(\lambda(u_j)) \ge f(2) = 2$. By the arbitrariness of the
admissible $k$-tuple, $\alpha_{k,2} \ge 2$, as desired.

We point out that equality forces $p=2$ and $\ell_{j} = \pi$ for the two indices
$j \in P$: in dimension $2$ the optimal configurations are exactly those for which
two of the sets $\omega_j$ are complementary open half-circles, while
$\omega_j = \S^1$ for every other index.

\end{proof}

\begin{remark}\label{rem: ex non-const}
In dimension $N=2$ (and hence also in higher dimension), one may consider the following competitor, having all non-constant components:
\[
u_1(\theta) = \begin{cases} \sin\left(\frac{k}{2(k-1)}\theta\right) & \text{if }x \in \left[0, \frac{2\pi}{k}(k-1)\right] \\
0 & \text{otherwise}, \end{cases}
\]
and
\[
u_2(\theta) = u_1\left(\theta-\frac{2\pi}{k}\right), \quad \dots \quad, u_k(\theta)= u_1\left(\theta- (k-1)\frac{2\pi}{k}\right).
\]
Then
\[
\prod_{j=1}^k u_j \equiv 0, \quad \text{and} \quad \lambda(u_1) =\cdots = \lambda(u_k) = \frac{k^2}{4(k-1)^2},
\]
so that
\[
\alpha_{k,2} \le \frac{k^2}{2(k-1)}.
\]
However, this number is greater than $2$, for every $k \ge 3$.
\end{remark}

The main result of this section is the following:

\begin{theorem}[ACF-type monotonicity formula]\label{thm: ACF}
Let $N \ge 2$, $k \ge 3$, $B_R \subset \R^N$, and $u_1,\dots,u_k \in H^1(B_R)$ be such that
\[
\Delta u_j \ge 0 \quad \text{and} \quad u_j \ge 0 \quad \text{in $B_R$},
\]
with
\[
\int_{B_R} u_1^2 \cdots u_k^2 \, dx = 0.
\]
For $x_0 \in B_R$ and $\nu := \alpha_{k,N}$, the function $r \mapsto J_{\nu}(u_1,\dots,u_k,x_0,r)$ is monotone non-decreasing.
\end{theorem}

\begin{proof}
The function $J(r):= J_{\alpha_{k,N}}(u_1,\dots,u_k,x_0,r)$ is absolutely continuous in $r$, and hence a.e. $r \in (0,\rho:= \dist(x_0,\pa B_R))$ is a Lebesgue point of $J$. Moreover, for a.e. $r \in (0,\rho)$ the restrictions $u_j|_{S_r(x_0)}$ are functions in $H^1(S_r(x_0))$, and
\[
\int_{S_r(x_0)} u_1^2\cdots u_k^2\,d\sigma = 0.
\]
We compute the derivative of $J$ with respect to the radius, denoted by $J'$, in any point $r$ for which both the above properties are satisfied, and in addition Lemma \ref{lem: ACF pre} holds, and verify that $J'(r) \ge 0$.

We suppose that $I(u_j,x_0,r)>0$ for every $j$, otherwise the fact that $J'(r) \ge 0$ follows simply from the non-negativity of $J$.

Let $u_{j,x_0,r}(\cdot)=u_j(x_0 + r \, \cdot)$. By assumption $\int_{\S^{N-1}} u_{1,x_0,r}^2 \cdots u_{k,x_0,r}^2 \, d\sigma = 0$ for a.e. $r$, and, by Lemma \ref{lem: ACF pre}, we have that
\begin{align*}
\frac{J'(r)}{J(r)} &=  \sum_{j=1}^k\frac{I'(u_j,x_0,r)}{I(u_j,x_0,r)} -\frac{2\alpha_{k,N}}{r}  =  \sum_{j=1}^k \frac{ \int_{S_r(x_0)} |\nabla u_j|^2 |x|^{2-N}\, d\sigma}{\int_{B_r(x_0)} |\nabla u_j|^2 |x|^{2-N}\,dx}  -\frac{2\alpha_{k,N}}{r}\\
& \ge \frac2r \left( \sum_{j=1}^k \gamma(\lambda(u_{j,x_0,r}))-   \alpha_{k,N}\right) \ge 0,
\end{align*}
which is the desired result.
\end{proof}

\subsection{Perturbed monotonicity formula} In this subsection we generalize the previous monotonicity formula in order to deal with non-segregated solutions of a class of elliptic systems. Recall that, for $\mf{u}=(u_1,\dots,u_k)$, we denote by $\hat{\mf{u}}_i$ the vector obtained by erasing the $i$-th component from $\mf{u}$. We consider
\begin{equation}\label{subsys}
-\Delta u_i + u_i^{q_i} g_i(x,\hat{\mf{u}}_i) = 0, \quad u_i\,\ge  \,0\quad \text{in $\R^N$, $i=1,\dots,k$},
\end{equation}
where $q_i \ge 1$ for every $i$, under the following assumptions on $g_i \in C(\R^N \times ([0,+\infty))^{k-1}, [0,+\infty))$:
\begin{itemize}
\item[(H1)] $\bar g_i(\hat{\mf{t}}_i) := \inf_{x \in \R^N} g_i(x,\hat{\mf{t}}_i)$ is a continuous function, with the property that $\bar g_i(\hat{\mf{t}}_i) >0$ if $t_j>0$ for every $j\neq i$, and $\bar g_i(\hat{\mf{t}}_i) =0$ if at least one component of $\hat{\mf{t}}_i$ vanishes. Even more, we suppose that $g_i(x,\hat{\mf{t}}_i) = 0$ for every $x \in \R^N$, if one component of $\hat{\mf{t}}_i$ vanishes.
\item[(H2)] For every $x \in \R^N$, $g_i(x,\,\cdot)$ is monotone non-decreasing in all its variables.
\end{itemize}
A prototypical example is
\[
g_i(x,\hat{\mf{t}}_i) = b(x) \prod_{\substack{j=1\\ j \neq i}}^k  t_j^{p_{ij}}, \quad \text{with $\inf_{\R^N} b>0$ and $p_{ij} >0$}.
\]

For $\mf{u}$ solving \eqref{subsys}, $x_0 \in \R^N$ and $r>0$, we use the following notation
\[
\tilde I_i(\mf{u},x_0,r) := \int_{B_r(x_0)} \left( |\nabla u_i|^2 + u_i^{q_i+1}g_i(x,\hat{\mf{u}}_i) \right) |x-x_0|^{2-N}\,dx.
\]
\begin{theorem}[Perturbed monotonicity formula]\label{thm: acf per}
Let $\mf{u} \in H^1_{\loc}(\R^N,\R^k) \cap C(\R^N,\R^k)$ satisfy \eqref{subsys}, with $q_1,\dots,q_k \ge 1$, and $g_i:\R^N \times ([0,+\infty))^{k-1} \to [0,+\infty)$ continuous and satisfying (H1) and (H2). For $x_0 \in \R^N$, $\nu := \alpha_{k,N}$ and any $\eps>0$, there exists $\bar r >0$ such that the function
\[
r \mapsto \tilde J_{\nu-\eps}(\mf{u},x_0,r)  := \frac{1}{r^{2(\nu-\eps)}} \prod_{j=1}^k \tilde I_j(\mf{u},x_0,r)
\]
is monotone non-decreasing for $r>\bar r$.
\end{theorem}

For the proof, we start with an estimate similar to the one in Lemma \ref{lem: ACF pre}. We introduce
\[
\Lambda_i(x_0,r)  = \frac{r^2 \int_{S_r(x_0)} \left(|\nabla_\theta u_i|^2 + u_i^{q_i+1} g_i(x,\hat{\mf{u}}_i)\right)\,d\sigma}{ \int_{S_r(x_0)} u_i^2\,d\sigma}, \quad i=1,\dots,k,
\]
which are well defined for a.e. $r$.
\begin{lemma}\label{lem: pre ACF per}
In the previous setting, for every $x_0 \in \R^N$ and a.e. $r>0$
\[
\int_{S_r(x_0)} \left(|\nabla u_i|^2 + u_i^{q_i+1} g_i(x,\hat{\mf{u}}_i)\right)|x-x_0|^{2-N}\,d\sigma \ge \frac{2 \gamma(\Lambda_i(x_0,r))}{r} \tilde I_i(\mf{u},x_0,r).
\]
\end{lemma}

The proof is analogous to that of Lemma \ref{lem: ACF pre}, we refer the interested reader to \cite[Lemma 5.1]{ST15} (see also \cite[Lemma 7.3]{CTV05} or \cite[Lemma 2.5]{NTTV10}) for the details.

\begin{proof}[Proof of Theorem \ref{thm: acf per}] The proof is similar to the one of \cite[Lemma 5.2]{ST15} (see also \cite[Lemma 7.3]{CTV05}, \cite[Lemma 2.5]{NTTV10}).
Without loss of generality, we fix $x_0=0$. As in the proof of Theorem \ref{thm: ACF}, thanks to Lemma \ref{lem: pre ACF per}, we have
\[
\frac{\tilde J'(r)}{\tilde J(r)}\ge \frac2{r}\left(\sum_{i=1}^k\gamma(\Lambda_i(0,r)) -(\alpha_{k,N} - \eps) \right),
\]
and the conclusion follows if we show that the right hand side is nonnegative for $r$ sufficiently large. Suppose by contradiction that this is not true: then there exists $r_n \to +\infty$ such that
\begin{equation}\label{cfp}
\sum_{i=1}^k \gamma(\Lambda_i(0,r_n)) < \alpha_{k,N} - \eps,
\end{equation}
and, in particular, $\{\Lambda_i(0,r_n)\}$ ($i=1,\dots,k$) are bounded sequences. Let
\[
u_{i,n}(x) = \frac{u_i(r_n x)}{\left( \frac1{r_n^{N-1}} \int_{S_{r_n}} u_i^2\,d\sigma\right)^\frac12}.
\]
We have that
\begin{align*}
 \int_{S_{1}} |\nabla_\theta u_{i,n}|^2 \,d\sigma \le \Lambda_i(0,r_n), \\
\end{align*}
so that $\{\mf{u}_n\}$ is bounded in $H^1(S_1,\R^k)$; moreover
\begin{align*}
\int_{S_1} u_{1,n}^{q_1+1} \bar g_1 \bigg( \Big( \frac1{r_n^{N-1}} \int_{S_{r_n}} & u_2^2\Big)^\frac12 u_{2,n}, \dots, \Big( \frac1{r_n^{N-1}} \int_{S_{r_n}} u_k^2 \Big)^\frac12 u_{k,n}   \bigg) d\sigma \\
&\le \frac{r_n^2 \int_{S_{r_n}} u_1^{q_1+1} g_1(x,\hat{\mf{u}}_1) \,d\sigma}{\int_{S_{r_n}} u_1^2 \,d\sigma} \cdot \frac{1}{r_n^2 \left( \frac1{r_n^{N-1}}\int_{S_{r_n}} u_1^2 \,d\sigma\right)^{\frac{q_1-1}2}} \\
&  \le \frac{\Lambda_1(0,r_n)}{r_n^2 \left( \frac1{r_n^{N-1}}\int_{S_{r_n}} u_1^2 \,d\sigma \right)^{\frac{q_1-1}2}} \to 0
\end{align*}
as $n \to \infty$, where we used assumption (H1) and the subharmonicity of $u_1$, which ensures that
\[
\frac1{r_n^{N-1}}\int_{S_{r_n}} u_{1}^2 \,d\sigma \ge C u_{1}^2(0) >0 \qquad \forall n.
\]
Similar estimates also hold, with the index $1$ replaced by any $i=2,\dots,k$. Therefore, we deduce that up to a subsequence $\mf{u}_n \weak \tilde{\mf{u}}$ weakly in $H^1(S_1)$, strongly in $L^2(S_1)$ (so that in particular $\tilde u_i \not \equiv 0$ for every $i$), and $\mathcal{H}^{N-1}$-almost everywhere, where $\tilde u_1 \cdots \tilde u_k = 0$ $\mathcal{H}^{N-1}$-a.e. on $S_1$: indeed, the previous estimates, Fatou's lemma, and the assumptions on $g_i$ ensure that
\[
\begin{split}
\int_{S_1} \tilde u_1^{q_1+1} \bar g_1(\delta \hat{\tilde{\mf{u}}}_i )\,d\sigma & \le \liminf_{n \to \infty} \int_{S_1} u_{1,n}^{q_1+1} \bar g_1(\delta \hat{\mf{u}}_{i,n}) \,d\sigma \\
& \le \liminf_{n \to \infty} \int_{S_1} u_{1,n}^{q_1+1} \bar g_1 \bigg( \Big( \frac1{r_n^{N-1}} \int_{S_{r_n}}  u_2^2\Big)^\frac12 u_{2,n}, \dots, \Big( \frac1{r_n^{N-1}} \int_{S_{r_n}} u_k^2 \Big)^\frac12 u_{k,n}   \bigg) d\sigma = 0,
\end{split}
\]
with $\delta=\min_i u_i^2(0)$, so that in $\mathcal{H}^{N-1}$-a.e. point of $S_1$ one component $\tilde u_i$ must vanish.

Coming back to \eqref{cfp}, by the definitions of $\alpha_{k,N}$ and $\gamma$, and using the fact that $\tilde u_i \not \equiv 0$, we obtain
\[
\begin{split}
\alpha_{k,N} \le \sum_{j=1}^k\gamma(\lambda(\tilde u_{j})) \le \liminf_{n \to \infty} \sum_{j=1}^k  \gamma(\Lambda_{j}(0,r_n)) \,\le\, \alpha_{k,N}-\eps,
\end{split}
\]
which is a contradiction.
\end{proof}

\section{Liouville-type theorems}\label{sec: liou}

The validity of the ACF monotonicity formulae proved in Section \ref{sec: ACF} allows us to obtain some nonexistence results. We first recall some preliminaries.

\subsection{Essentially known results}

In this section we recall some useful preliminary results which are probably known and will be used repeatedly in the next sections. At first, we recall the following estimate proved in \cite[Lemma 4.4]{CTV05}.

\begin{lemma}\label{lem: decay}
Let $R>0$, and let $u \in H^1(B_{2R})$ satisfy
\[
\begin{cases}
-\Delta u \le -M u & \text{in $B_{2R}$} \\
u \ge 0 & \text{in $B_{2R}$} \\
u \le A & \text{on $\pa B_{2R}$},
\end{cases}
\]
for some $A, M>0$. Then there exists $C>0$ such that
\[
\|u\|_{L^\infty(B_R)} \le C A e^{-\frac{R \sqrt{M}}{2}}.
\]
\end{lemma}

Now we state some Liouville-type theorems for globally H\"older continuous solutions of certain elliptic problems. Recall that, for some $\alpha \in (0,1)$, we say that $v$ is globally $\alpha$-H\"older continuous in $\R^N$ if the seminorm $[v]_{C^{0,\alpha}(\R^N)}$ is finite (no limitation on the $L^\infty$ norm of $v$ is required).

\begin{proposition}\label{prop: Liou eq}
Let $\alpha \in (0,1)$, and let $w$ be a globally $\alpha$-H\"older continuous function in $\R^N$. Suppose moreover that one of the following equations is satisfied by $w$:
\begin{itemize}
\item[($i$)] either $\Delta w=\lambda$ in $\R^N$ for some $\lambda \in \R$;
\item[($ii$)] or $\Delta w = \lambda w$ with $\lambda >0$ in $\R^N$.
\end{itemize}
Then $w$ must be constant and, in case ($ii$), $w \equiv 0$.
\end{proposition}
\begin{proof}
($i$) The case $\lambda=0$ is the classical Liouville theorem for harmonic functions. In general, for a fixed $h \in \R^N$, let $\varphi(x):= w(x+h)-w(x)$, defined for $x \in \R^N$. Plainly, $\varphi$ is harmonic and globally H\"older continuous. Therefore, it must be equal to a constant $c \in \R$, which implies that
\[
w(x+h) = w(x) + c \quad \forall x \in \R^N.
\]
By taking the derivative of both sides with respect to $x_i$, we infer that
\[
\pa_{x_i} w(x+h) = \pa_{x_i} w(x) \quad \forall x \in \R^N.
\]
Since this holds true for every $h \in \R^N$, all the partial derivatives $\pa_{x_i} w$ must be constant on $\R^N$. On the other hand, the global H\"older continuity implies that $w$ has at most strictly sublinear growth at infinity, and hence such partial derivatives must vanish everywhere.\\
($ii$) Notice that $\Delta w^+ \ge \lambda w^+$ in $\R^N$, by Kato's inequality. Since $0 \le w^+(x) \le C(1+|x|^\alpha)$ on $\R^N$, by Lemma \ref{lem: decay} we deduce that for every $R>0$ there exists $C>0$ (independent of $R$) such that
\[
w^+(x) \le C (1+(2R)^\alpha) e^{-\frac12 R \sqrt{\lambda}} \quad \text{for every $x  \in B_R$}.
\]
By taking the limit as $R \to \infty$, we deduce that $w^+ \equiv 0$. In a similar way, we can also show that $w^- \equiv 0$.
\end{proof}

Now we recall some Liouville-type theorems for systems with $2$ components.

\begin{proposition}\label{prop: Liou sys 2}
Let $\alpha \in (0,1)$, and let $u,v \in H^1_{\loc}(\R^N) \cap C(\R^N)$ be globally $\alpha$-H\"older continuous functions in $\R^N$. Suppose moreover that one the following equation is satisfied by $(u,v)$:
\begin{itemize}
\item[($i$)] either
\[
\begin{cases}
\Delta u= \lambda \, u\, v^2   \\
\Delta v= \lambda \, u^2 \, v, \end{cases} \qquad u, v \ge 0  \qquad\text{in $\R^N$,}
\]
with $\lambda>0$;
\item[($ii$)] or
\beq\label{sys2seg}
\begin{cases}
\Delta u \ge 0   \\
\Delta v \ge 0 \\
u \, v \equiv 0,
\end{cases}
\qquad u, v \ge 0, \qquad  \text{in $\R^N$}.
\eeq
\end{itemize}
Then one component of $(u,v)$ must vanish; moreover, in case ($i$), the other must be constant.
\end{proposition}

These results are proved in \cite[Section 2]{NTTV10}.

\subsection{Liouville theorem for partially segregated functions} We aim to obtain some Liouville-type theorems similar to Proposition \ref{prop: Liou sys 2} for systems modeling partial segregation. The first main result of this section is the following:

\begin{theorem}\label{thm: liou seg}
Let $k \ge 3$ and $N \ge 2$ be positive integers, and let $\alpha \in (0,\alpha_{k,N}/k)$. For $i=1,\dots,k$, let $u_i \in H^1_{\loc}(\R^N) \cap C(\R^N)$ be globally $\alpha$-H\"older continuous functions in $\R^N$, such that
\[
-\Delta u_i \le 0 \quad \text{and} \quad u_i \ge 0 \quad \text{in $\R^N$},
\]
and moreover the partial segregation condition holds:
\[
\prod_{j=1}^k u_j  \equiv 0 \quad \text{in $\R^N$}.
\]
Then at least one function $u_j$ must be constant.
\end{theorem}

\begin{proof}
By using the monotonicity formula we show that this provides a contradiction, following the same strategy developed in \cite[Proposition 2.2]{NTTV10}.

Let us suppose by contradiction that $u_i$ is non-constant in $\R^N$ for every $i=1,\dots,k$. This implies that there exists $x_0 \in \R^N$ and $\bar r>0$ sufficiently large such that $I(u_i,x_0,\bar r)>0$ for every $i$, so that, by the monotonicity formula in Theorem \ref{thm: ACF},
\[
\prod_{j=1}^k I(u_{j},x_0,r) \ge C_1 r^{2\alpha_{k,N}} \qquad \text{for $r>\bar r$},
\]
with $C_1>0$. Let now $r> \bar r$, and consider a radial smooth cutoff function $\eta$ such that $0 \le \eta \le 1$, $\eta =1$ in $B_r(x_0)$, $\eta = 0$ in $\R^N \setminus B_{2r}(x_0)$, and $|\nabla \eta| \le C/r$. By testing the inequality satisfied by $u_i$ with $\eta^2 \Phi_\delta(x-x_0) u_i$ (with $\Phi_{\delta}$ defined in \eqref{reg_fun}), and proceeding as in \cite[Proof of Proposition 2.2]{NTTV10}, we obtain
\[
\begin{split}
I(u_i,x_0,r) & \le C_2 r^{2\alpha},
\end{split}
\]
with $C_2>0$. By combining the former inequalities, we conclude that for $r>\bar r$
\[
C_1 r^{2\alpha_{k,N}}  \le C_2^{k} r^{2k\alpha},
\]
which is a contradiction for large $r$ since $k \alpha < \alpha_{k,N}$.
\end{proof}

In the next section, we will use the following variant.

\begin{corollary}\label{cor: liou seg}
Let $N \ge 2$ be a positive integer, and let $\alpha \in (0,\alpha_{3,N}/3)$. For $i=1,2,3$, let $u_i \in H^1_{\loc}(\R^N) \cap C(\R^N)$ be globally $\alpha$-H\"older continuous functions in $\R^N$, such that
\[
-\Delta u_i \le 0 \quad \text{and} \quad u_i \ge 0 \quad \text{in $\R^N$},
\]
and moreover the partial segregation condition holds:
\[
u_1 \, u_2 \, u_3 \equiv 0 \quad \text{in $\R^N$}.
\]
Then at least one component $u_i$ vanishes identically.
\end{corollary}

\begin{proof}
By Theorem \ref{thm: liou seg}, at least one component $u_i$ is constant, $u_i \equiv c \in \R$ on $\R^N$. If $c=0$, then the proof is complete. Suppose instead that $c>0$. Then the remaining components, which we denote by $u$ and $v$, satisfy system \eqref{sys2seg}, and hence by Proposition \ref{prop: Liou sys 2} one of them must vanish identically.
\end{proof}

\subsection{Liouville theorem for solutions to certain elliptic systems}\label{sub: liou}

The second main result of this section is the following:

\begin{theorem}\label{thm: liou sub}
Let $k \ge 3$ and $N \ge 2$ be positive integers, and, for $i=1,\dots,k$, let $u_i \in H^1_{\loc}(\R^N) \cap C(\R^N)$ be nonnegative functions satisfying a system of type \eqref{subsys} in $\R^N$, where $q_1,\dots,q_k \ge 1$, and  the coupling terms $g_i$ satisfy assumptions (H1) and (H2). Assume moreover that
\[
0\le u_i(x) \le C(1+|x|^{\alpha}) \quad \text{for every $x \in \R^N$},
\]
with $\alpha <\alpha_{k,N}/k$. Then at least one component $u_i$ vanishes identically.
\end{theorem}

\begin{proof}
Suppose by contradiction that all the components are non-trivial. By the maximum principle, we have that $u_i>0$ in $\R^N$ for every $i$. Let $0<\eps<\alpha_{k,N}-k\alpha$. Then, by Theorem \ref{thm: acf per}, there exist $C, \bar r>0$ such that
\begin{equation}\label{below'}
\prod_{j=1}^k \tilde I_j(\mf{u}, 0,r) \ge Cr^{2(\alpha_{k,N}-\eps)}
\end{equation}
for $r> \bar r$. On the other hand, let $\eta$ be a cutoff function as in the proof of Theorem \ref{thm: liou seg}, and let $\Phi_{\delta}$ be defined in \eqref{reg_fun}. By testing the inequality satisfied by $u_i$ by $\eta^2 \Phi_{\delta} u_i$, we obtain (as in \cite[Proposition 2.6]{NTTV10} or \cite[Theorem 3.4]{ST23})
\[
\tilde I_i(\mf{u},0,r) \le C r^{2\alpha}
\]
for $r> 1$ and for every $i$. This gives a contradiction with \eqref{below'} for $r$ large, since $k\alpha <\alpha_{k,N}-\eps$.
\end{proof}

\section{Proof of Theorem \ref{thm: holder bounds}}\label{sec: hold bounds}

The proof of Theorem \ref{thm: holder bounds} is the content of the present section. The general strategy is inspired by \cite{CTV05, NTTV10, STTZ16, TVZ16, Wa}. However, we have to substantially modify the argument in order to deal with our interaction term modeling partial segregation.

Without loss of generality, we suppose that $\Omega \supset B_3$, and we aim to prove the uniform H\"older bound in $B_1$. We know that
\[
\sup_{i=1,2,3} \|u_{i,\beta}\|_{L^\infty(B_2)} \le C <+\infty
\]
independently of $\beta$. Let $\eta \in C^1_c(\R^N)$ be a radially decreasing cut-off function such that
$\eta \equiv 1$ in $\overline{B_1}$, $\eta \equiv 0$ in $\R^N \setminus B_2$.

We fix $\alpha \in (0,\alpha_{3,N}/3)$, and aim to prove that the family $\{\eta \mf{u}_\beta\}_{\beta >1}$ admits a uniform bound on the $\alpha$-H\"older semi-norm, that is, there exists $C > 0$, independent of $\beta$, such that
\begin{equation}\label{lip scaling}
\sup_{i=1,2,3} \sup_{\substack{x \neq y \\ x,y \in B_3}}    \frac{ |(\eta u_{i,\beta})(x)-(\eta u_{i,\beta})(y)|}{|x-y|^{\alpha}} \leq C.
\end{equation}
Since $\eta\equiv 1$ in $\overline{B_1}$, once \eqref{lip scaling} is proved, Theorem \ref{thm: holder bounds} follows.

Let us assume by contradiction that there exists a sequence $\beta_n \to +\infty$ and a corresponding sequence $\{\mf{u}_n:= \mf{u}_{\beta_n}\}$ such that
\[
    L_n := \sup_{i = 1, 2, 3} \sup_{\substack{x \neq y \\ x,y \in B_3}} \frac{ |(\eta u_{i,n})(x)-(\eta u_{i,n})(y)|}{|x-y|^{\alpha}} \to \infty \qquad \text{as $n \to +\infty$.}
\]
Since, for $\beta_n$ fixed, the functions $\mf{u}_{i,n}$ are smooth, we may assume that up to a relabelling the supremum is achieved for $i = 1$ and at a pair of points $x_n, y_n \in \overline{B_{2}}$, with $x_n \neq y_n$ . As $\{\mf{u}_n\}$ is uniformly bounded in $L^\infty(B_2)$, we have that
\[
|x_n-y_n|^{\alpha}= \frac{|(\eta u_{1,n})(x_n)-(\eta u_{1,n})(y_n)|}{L_n} \le \frac{C}{L_n} \to 0
\]
as $n \to \infty$.

\subsection{Blow-up analysis}\label{subsec:blowup}

The contradiction argument is based on two blow-up sequences:
   \[
    v_{i,n}(x) := \eta(x_n) \frac{u_{i,n}(x_n + r_n x)}{L_n r_n^{\alpha}} \quad \text{and} \quad \bar{v}_{i,n}(x) := \frac{(\eta u_{i,n})(x_n + r_n x)}{L_n r_n^{\alpha}},
\]
where
\[
r_n := |x_n -y_n| \to 0^+,
\]
both defined on the scaled domain
\[
\frac{\Omega - x_n}{r_n} \supset \frac{B_3-x_n}{r_n}=: \Omega_n \supset B_{1/r_n},
\]
which exhaust $\R^N$ as $\beta \to \infty$.

Now, the function $\bar{\mf{v}}_{n}$ is the one for which the H\"older quotient is normalized (see Lemma \ref{lem: basic prop}-(1) ahead), however it satisfies a rather complicated system. On the other hand, $\mf{v}_{n}$ satisfies a simple system related to \eqref{P beta}, but it may not have bounded seminorm. We will also check that both blow-up functions have (locally) comparable $L^\infty$ norms and oscillations, and this allows to interchange information from one function to the other. This idea was first used in the context of singularly perturbed elliptic systems with full competition by K. Wang \cite{Wa} (see also \cite{STTZ16}).

Basic properties of the blow-up sequences are collected in the following lemma.

\begin{lemma}\label{lem: basic prop}
In the previous setting, it results that:
\begin{itemize}
\item[($i$)] the sequence $\{\bar{\mf{v}}_n\}$ has uniformly bounded $\alpha$-H\"older semi-norm in $\Omega_n$, and in particular
\[
\sup_{i=1,2,3} \sup_{\substack{x \neq y \\ x,y \in \Omega_n}}    \frac{ |\bar v_{i,n}(x)-\bar v_{i,n}(y)|}{|x-y|^{\alpha}} = \frac{ |\bar v_{1,n}(0)-\bar v_{1,n}\left(\frac{y_n-x_n}{r_n}\right)|}{\left|\frac{y_n-x_n}{r_n}\right|^{\alpha}} = 1
\]
for every $n$.
\item[($ii$)] $v_{i,n}$ is a solution of
\begin{equation}\label{system blow-up}
\Delta v_{i,n} = M_n v_{i,n} \prod_{j \neq i} v_{j,n}^{2}, \quad v_{i,n}>0 \qquad \text{in $\Omega_n$},
\end{equation}
where
\[
 M_n:= \beta_n r_n^{2+4\alpha} \left(\frac{L_n}{\eta(x_n)}\right)^{4}.
\]
\item[($iii$)] for every open set $\omega \ssubset \R^N$ we have
\[
\sup_{\overline{\omega}} |\mf{v}_n - \bar{\mf{v}}_n | \to 0 \qquad \text{as $n \to \infty$}.
\]
\item[($iv$)] for every compact set $K \subset \R^N$, and for every $n$ so large that $\Omega_n \supset K$, we have
\[
|v_{i,n}(x) - v_{i,n}(y)| \le o_n(1) + |x-y|^\alpha \quad \text{for every $x, y \in K$},
\]
where $o_n(1) \to 0$ uniformly on $x,y \in K$; in particular $\{v_{i,n}\}$ has uniformly bounded oscillation in any compact set.
\item[($v$)] $\mf{v}_n$ is a minimizer of \eqref{system blow-up} with respect to variations with compact support, namely for every $\Omega' \ssubset \R^N$
\[
J_{M_n}(\mf{v}_n,\Omega') \le J_{M_n}(\mf{v}_n+ \bs{\varphi},\Omega') \quad \forall \bs{\varphi} \in H^1_0(\Omega', \R^3),
\]
for sufficiently large $n$, where $J_{M_n}(\mf{u}, \Omega)$ is defined as in \eqref{def functional}.
\end{itemize}
\end{lemma}
\begin{proof}
The proof of points ($i$), ($ii$) and ($v$) is straightforward. As far as ($iii$) is concerned, since $\eta$ is globally Lipschitz continuous with Lipschitz constant $l$, and $\{u_{i,n}\}$ is uniformly bounded in $K$, we have
\[
|v_{i,n}(x) - \bar v_{i,n}(x)|= \frac{|u_{i,n}(x_n+r_n x)|}{L_n r_n^{\alpha}} |\eta(x_n)-\eta (x_n+r_n x)| \le \frac{l M r_n^{1-\alpha}}{L_n} |x|,
\]
where we recall that $\|u_{i,n}\|_{L^\infty(B_3)} \le M$ for every $i$ and $n$. Finally, for ($iv$) we use point ($iii$) and the uniform H\"older boundedness of the sequence $\{\bar{\mf{v}}_n\}$.
\end{proof}

In the next statement we collect some useful properties which take place when one component is bounded at $0$.

\begin{lemma}\label{lem: if i bdd}
Suppose that $\{v_{i,n}(0)\}$ is bounded. Then there exists a function $v_i \in C^{0}(\R^N) \cap H^1_{\loc}(\R^N)$, globally $\alpha$-H\"older continuous, such that, up to a subsequence:
\begin{itemize}
\item[($i$)] $v_{i,n}, \bar v_{i,n} \to v_i$
locally uniformly on $\R^N$;
\item[($ii$)] $v_i$ is non-constant in $B_1$ if $i=1$;
\item[($iii$)] for every $\omega \ssubset \R^N$ compact there exists $C>0$ such that
\[
M_n \int_{\omega} \Big(\prod_{j \neq i} v_{j,n}^2 \Big) v_{i,n}\,dx \le C
\]
\item[($iv$)] $v_{i,n} \to v_i$ strongly in $H^1_{\loc}(\R^N)$.
\end{itemize}
\end{lemma}

\begin{proof}
Point ($i)$ follows directly from Arzel\`a--Ascoli theorem and Lemma \ref{lem: basic prop}. Point ($ii$) is a consequence of the definition of $\{\bar{\mf{v}}_n\}$. Concerning point ($iii$), let us test the equation for $v_{i,n}$ with a cutoff function $\varphi \in C^\infty_c(\R^N)$. We obtain
\[
M_n \int_{\Omega_n} \Big(\prod_{j \neq i} v_{j,n}^2 \Big) v_{i,n} \varphi \,dx = \int_{\Omega_n} v_{i,n} \Delta \varphi \,dx,
\]
whence the desired estimate follows. Finally, for point ($iv$), let us test the equation for $v_{i,n}$ with $v_{i,n} \varphi^2$, where $\varphi \in C^\infty_c(\R^N)$ is an arbitrary nonnegative cutoff function. We obtain
\[
\int_{\Omega_n} \Big(|\nabla v_{i,n}|^2 \varphi^2 +2v_{i,n} \varphi \nabla v_{i,n}\cdot \nabla \varphi +M_n \varphi^2 \prod_{j=1}^3 v_{j,n}^2 \Big) dx = 0.
\]
This implies that
\[
\int_{\Omega_n} |\nabla v_{i,n}|^2 \varphi^2\,dx \le 4 \int_{\Omega_n}  v_{i,n}^2 |\nabla \varphi|^2\,dx,
\]
which entails the boundedness of $\{v_{i,n}\}$ in $H^1_{\loc}(\R^N)$, and hence up to a subsequence $v_{i,n} \weak v_i$ weakly in $H^1_{\loc}(\R^N)$. To show the strong convergence, we test the equation for $v_{i,n}$ with $(v_{i,n}-v_i) \varphi$, where $\varphi \in C^\infty_c(\R^N)$ is an arbitrary nonnegative cutoff function, and observe that
\[
\int_{\Omega_n} \varphi \nabla v_{i,n} \cdot \nabla (v_{i,n}-v_i)  \, dx \le \|v_{i,n}-v_i\|_{L^\infty(\mathrm{supp}(\varphi))} \int_{\Omega_n} \Big( |\nabla v_{i,n} \cdot \nabla \varphi| + M_n \varphi \, v_{i,n} \prod_{j \neq i} v_{j,n}^2 \Big)dx.
\]
The previous points ensure that the right hand side tends to $0$ as $n \to \infty$. Therefore
\[
\lim_n \int_{\Omega_n} \varphi \left( |\nabla v_{i,n}|^2- |\nabla v_i|^2\right) dx =\lim_{n} \int_{\Omega_n} \varphi \nabla v_{i,n} \cdot \nabla (v_{i,n}-v_i)\,dx =0,
\]
and the proof is complete.
\end{proof}

The previous lemma shows that the boundedness of $\{v_{i,n}(0)\}$ entails several consequences. On the other hand, we do not know whether this property holds or not. If not, we will often use the following statement.

\begin{lemma}\label{lem: if i unb}
Let
\[
\mf{w}_n(x):= \mf{v}_n(x)-\mf{v}_n(0), \qquad \bar{\mf{w}}_n(x):= \bar{\mf{v}}_n(x)-\bar{\mf{v}}_n(0).
\]
Then there exists a function $\mf{w}$, globally $\alpha$-H\"older continuous in $\R^N$, such that $\mf{w}_n, \bar{\mf{w}}_n \to \mf{w}$ locally uniformly in $\R^N$. Moreover, $w_1$ is non-constant.
\end{lemma}
\begin{proof}
The local uniform convergence $\bar{\mf{w}}_n \to \mf{w}$ follows directly by the Arzel\`a--Ascoli theorem and the uniform bound on the $C^{0,\alpha}$ seminorm of $\bar{\mf{w}}_n$. Moreover the $\alpha$-H\"older seminorm of $\mf{w}$ is equal to $1$, and $w_1$ is non-constant. The fact that also $\mf{w}_n$ converges to $\mf{w}$ is a consequence of Lemma \ref{lem: basic prop}-($iii$).
\end{proof}

In the next lemmas we focus on what happens if at least one component is unbounded. We will reach a contradiction in all the possible cases.

\begin{lemma}\label{lem: 1 a infty}
Suppose that up to a subsequence $v_{i,n}(0) \to +\infty$. Then for every $r>0$ there exists $C>0$ such that
\[
M_n v_{i,n}(0) \int_{B_r} \prod_{j \neq i} v_{j,n}^2 \,dx \le C.
\]
\end{lemma}

\begin{proof}
We define
\[
\begin{split}
H_{i,n}(r) &:= \frac1{r^{N-1}} \int_{S_r} v_{i,n}^2\,d\sigma, \\
E_{i,n}(r) & := \frac1{r^{N-2}} \int_{B_r} \Big( |\nabla v_{i,n}|^2 + M_n \prod_{j=1}^3 v_{j,n}^2\Big)\,dx
\end{split}
\]
By multiplying the equation of $v_{i,n}$ by $v_{i,n}$, and integrating on $B_r$, we see that
\[
\int_{B_r} \Big(|\nabla v_{i,n}|^2 + M_n \prod_{j=1}^3 v_{j,n}^2\Big)dx = \int_{S_r} v_{i,n} \pa_\nu v_{i,n}\,d\sigma,
\]
whence
\[
H_{i,n}'(r) = \frac2{r} E_{i,n}(r).
\]
Therefore, integrating from $r/2$ to $r$ we obtain
\beq\label{H_i' e E_i}
H_{i,n}(r) - H_{i,n}\left(\frac{r}{2}\right) = \int_{r/2}^r \frac2{s} E_{i,n}(s)\,ds.
\eeq
Now, the left hand side can be estimated from above thanks to Lemma \ref{lem: basic prop} - properties ($i$) and ($iv$):
\[
H_{i,n}(r) - H_{i,n}\left(\frac{r}{2}\right) = \int_{S_1} \left( v_{i,n}^2(ry)- v_{i,n}^2 \left(\frac{r}{2}y\right)\right)dy \le C(r) \left( v_{i,n}(0)+1\right) \le C_1(r) v_{i,n}(0).
\]
Moreover, the right hand side of \eqref{H_i' e E_i} can be estimated from below as follows:
\[
\begin{split}
\int_{r/2}^r \frac2{s} E_{i,n}(s)\,ds & \ge \frac{r}{2} \min_{s \in \left[\frac{r}{2},r\right]} \frac{2}{s} E_{i,n}(s) \\
& \ge r M_n \min_{s \in \left[\frac{r}{2},r\right]} \frac{1}{s^{N-1}} \int_{B_s} \prod_{j=1}^3 v_{j,n}^2\,dx\\
& \ge C_2(r) M_n v_{i,n}^2(0) \int_{B_{r/2}} \frac{v_{i,n}^2}{v_{i,n}^2(0)} \prod_{j \neq i} v_{j,n}^2\,dx.
\end{split}
\]
Since $v_{i,n}(0) \to +\infty$ and the oscillation of $v_{i,n}$ is locally uniformly bounded, we have that $v_{i,n}/v_{i,n}(0) \to 1$ uniformly on $B_{r/2}$ as $n \to \infty$. Hence, coming back to \eqref{H_i' e E_i}, we finally infer that
\[
M_n v_{i,n}^2(0) \int_{B_{r/2}} \prod_{j \neq i} v_{j,n}^2\,dx \le C(r) v_{i,n}(0),
\]
whence the conclusion follows.
\end{proof}

\begin{lemma}\label{lem: 2 a infty}
Suppose that up to a subsequence $v_{i_1,n}(0), v_{i_2,n}(0) \to +\infty$ for a couple of indices $i_1, i_2 \in \{1,2,3\}$. Then
\[
M_n v_{i_1,n}^2(0) \, v_{i_2,n}^2(0) \le C
\]
for a positive constant $C>0$ independent of $n$.
\end{lemma}

\begin{proof}
We argue by contradiction, and suppose that up to a subsequence $M_n v_{i_1,n}^2(0) v_{i_2,n}^2(0) \to +\infty$ as $n \to \infty$.

\smallskip

\emph{Step 1)} Let $\underline{i} \in \{1,2,3\} \setminus \{i_1,i_2\}$. For any $R>1$, we claim that $\|v_{\underline{i},n}\|_{L^\infty(B_{R})} \le C$ for a positive constant $C>0$ depending only on $R$.

\smallskip

First of all, since $v_{j,n}$ has locally bounded oscillation for each $i$ (Lemma \ref{lem: basic prop}), we have that $v_{j,n} \to +\infty$ locally uniformly for $j=i_1,i_2$. Moreover, there exists $C_R>0$ depending on $R$ but not on $n$ such that
\beq\label{def I_n}
I_n := M_n \inf_{B_{R}}v_{i_1,n}^2 \, v_{i_2,n}^2 \ge M_n \left(v_{i_1,n}(0)- C_R\right)^2 \left(v_{i_2,n}(0)- C_R\right)^2 \to +\infty.
\eeq
Let us test the equation of $v_{\underline{i},n}$ with $v_{\underline{i},n} \varphi^2$, with $\varphi \in C^\infty_c(\R^N)$ such that $0 \le \varphi \le 1$, $\varphi \equiv 1$ in $B_R$, and $\supp(\varphi) \subset B_{2R}$: we obtain
\[
\int_{\Omega_n} \Big(|\nabla v_{\underline{i},n}|^2 \varphi^2 +2v_{\underline{i},n} \varphi \nabla v_{\underline{i},n}\cdot \nabla \varphi +M_n \varphi^2 \prod_{j=1}^3 v_{j,n}^2 \Big) dx = 0,
\]
whence by standard computations
\beq\label{stima prod}
\begin{split}
C_R \, I_n \inf_{B_R} v_{\underline{i},n}^2 & \le M_n \int_{B_R} \prod_{j=1}^3 v_{j,n}^2 \,dx  \le \int_{\Omega_n} \Big(\frac12 |\nabla v_{\underline{i},n}|^2 \varphi^2 +M_n \varphi^2 \prod_{j=1}^3 v_{j,n}^2 \Big) dx\\
& \le 2\int_{\Omega_n} v_{\underline{i},n}^2 |\nabla \varphi|^2\,dx
 \le C_R \sup_{B_{2R}} v_{\underline{i},n}^2,
\end{split}
\eeq
for some $C_R>0$. We also observe that $\inf_{B_R} v_{\underline{i},n} \ge \sup_{B_{2R}} v_{\underline{i},n} - C_R$, with $C_R>0$ (once again, we used the fact that $\mf{v}_n$ has locally bounded oscillation). Therefore, \eqref{stima prod} yields
\[
I_n \sup_{B_{2R}} v_{\underline{i},n}^2 \le C_R \Big(\sup_{B_{2R}} v_{\underline{i},n}^2 +  I_n \sup_{B_{2R}} v_{\underline{i},n}\Big),
\]
which, in view of \eqref{def I_n}, is possible only if $\sup_{B_{2R}} v_{\underline{i},n}^2$ is bounded, as claimed.

\smallskip

\emph{Step 2)} Conclusion of the proof.

\smallskip

By definition of $I_n$, we have that
\[
-\Delta v_{\underline{i},n} \le - I_n v_{\underline{i},n}
\quad \text{in $B_{R}$}.
\]
By Lemma \ref{lem: decay}, Step 1 of the present proof, and recalling that $I_n \to +\infty$, we infer that
\[
\sup_{B_{R/2}} v_{\underline{i},n} \le C \sup_{B_{R}} v_{\underline{i},n} e^{- \frac{R\sqrt{I_n}}{2}} \to 0
\]
as $n \to \infty$. Since $R>1$ was arbitrary, $v_{\underline{i},n} \to 0$ locally uniformly in the whole space $\R^N$. Together with Lemma \ref{lem: if i bdd}, this rules out the case $\underline{i}=1$. Thus, without loss of generality, we can suppose that $\underline{i}=3$, $i_1=1$, $i_2=2$.

Let now $R>1$, and $\bar x_n \in \overline{B_R}$ such that
\[
M_n v_{1,n}^2 v_{2,n}^2\,(\bar x_n) = I_n,
\]
where $I_n$ is defined in \eqref{def I_n}. Since $\{\mf{v}_n\}$ has locally bounded oscillation, while $v_{1,n}, v_{2,n} \to +\infty$ locally uniformly, we have that
\[
\begin{split}
M_n \big( v_{1,n} v_{2,n}(x)\big)^2 &\le M_n(v_{1,n}(\bar x_n)+C_R)^2(v_{2,n}(\bar x_n)+C_R)^2 \le 2 M_n \big( v_{1,n} v_{2,n}(\bar x_n)\big)^2 = 2I_n
\end{split}
\]
for every $x \in B_R$.
Using once again that $v_{1,n} \to +\infty$ locally uniformly, while $\{v_{3,n}\}$ is locally bounded, we deduce that there exists $C_R>0$ such that
\[
\begin{split}
|\Delta v_{1,n}(x)| & \le  M_n v_{1,n}^2(x) \prod_{j \neq 1} v_{j,n}^2(x)  \le M_n \big( v_{1,n} v_{2,n}(x)\big)^2 C_R e^{-C_R\sqrt{I_n}}   \le  C_R  I_n  e^{-C_R\sqrt{I_n}} \quad \forall x \in B_R
\end{split}
\]
for sufficiently large $n$. If we consider $w_{1,n}:= v_{1,n}-v_{1,n}(0)$, as in Lemma \ref{lem: if i unb}, this implies that $w_{1,n} \to w_1$ locally uniformly in $\R^N$, with $w_1$ globally H\"older continuous, non-constant, and harmonic. Indeed, $|\Delta w_{1,n}| = |\Delta v_{1,n}| \to 0$ locally uniformly in $\R^N$. By the classical Liouville theorem for harmonic functions, this is the desired contradiction.
\end{proof}

\begin{lemma}\label{lem: all infty}
It is not possible that $v_{i,n}(0) \to +\infty$ for every $i=1,2,3$.
\end{lemma}

\begin{proof}
By contradiction, we suppose that $v_{i,n}(0) \to +\infty$ for every $i$.
Then $v_{i,n}(z) \to +\infty$ for every $z \in \R^N$ as well; more precisely, since $v_{j,n}$ has locally bounded oscillation, we have that
\beq\label{760}
\frac{v_{j,n}}{v_{j,n}(0)} \to 1 \quad \text{uniformly on any compact set}.
\eeq
Thus, for any $r>0$, we have that
\[
\int_{B_r} \prod_{j \neq 1} \left(\frac{v_{j,n}}{v_{j,n}(0)}\right)^2 dx \to |B_r|
\]
as $n \to \infty$, and by Lemma \ref{lem: 1 a infty}
\beq\label{26824}
\begin{split}
M_n v_{1,n}(0) \prod_{j \neq 1} v_{j,n}^2(0) & \le C M_n v_{1,n}(0) \prod_{j \neq 1} v_{j,n}^2(0) \int_{B_r} \prod_{j \neq 1} \left(\frac{v_{j,n}}{v_{j,n}(0)}\right)^2 dx \\
&= C M_n v_{1,n}(0) \int_{B_r} \prod_{j \neq 1} v_{j,n}^2\,dx \le C,
\end{split}
\eeq
with $C$ independent of $n$; namely, the right hand side in the equation of $v_{1,n}$ is bounded at $0$. Actually, thanks to \eqref{760} and \eqref{26824}, we deduce that
\[
\begin{split}
\big| M_n v_{1,n}(z) v_{2,n}^2(z) &v_{3,n}^2(z)  - M_n v_{1,n}(0) v_{2,n}^2(0) v_{3,n}^2(0) \big| \\
& = M_n v_{1,n}(0) v_{2,n}^2(0) v_{3,n}^2(0) \left| 1- \frac{v_{1,n}(z)}{v_{1,n}(0)} \frac{v_{2,n}^2(z)}{v_{2,n}^2(0)} \frac{v_{3,n}^2(z)}{v_{3,n}^2(0)} \right| \to 0
\end{split}
\]
as $n \to \infty$, uniformly in $z \in B_r$, for every $r>0$. Therefore, there exists a constant $\lambda \ge 0$ such that, up to a subsequence,
\[
\Delta v_{1,n} \to \lambda \quad \text{locally uniformly in $\R^N$},
\]
By Lemma \ref{lem: if i unb}, $w_{1,n}(x) := v_{1,n}(x)-v_{1,n}(0)$ converges locally uniformly to a limit function $w_1$, non-constant and globally H\"older continuous, which moreover satisfies
\[
\Delta w_1 = \lambda \quad \text{in $\R^N$}.
\]
By Proposition \ref{prop: Liou eq}, this is the desired contradiction.
\end{proof}

\begin{lemma}\label{lem: 2 infty}
It is not possible that exactly two components $v_{i_1,n}(0), v_{i_2,n}(0) \to +\infty$.
\end{lemma}

\begin{proof}
In case $v_{i_1,n}(0), v_{i_2,n}(0) \to +\infty$ for a couple of indices $i_1,i_2 \in \{1,2,3\}$, we know by Lemma \ref{lem: 2 a infty} that
\[
M_n v_{i_1,n}^2(0) v_{i_2,n}^2(0) \le C.
\]
As in the previous proof, we have that $v_{j,n}/v_{j,n}(0) \to 1$ locally uniformly, for $j=i_1, i_2$, whence it follows that
\[
|M_n v_{i_1,n}^2 v_{i_2,n}^2 - M_n v_{i_1,n}^2(0) v_{i_2,n}^2(0)| \to 0
\]
uniformly on any compact set of $\R^N$. Therefore, there exists a constant $\lambda \ge 0$ such that, up to a subsequence, $M_n v_{i_1,n}^2 v_{i_2,n}^2 \to \lambda$ locally uniformly in $\R^N$. In turn, this implies that the last component $v_{i_3,n}$, which is bounded at $0$ (and hence in any compact set), converges locally uniformly to a function $v_{i_3}$ which is nonnegative, globally H\"older continuous, and satisfies $\Delta v_{i_3} = \lambda v_{i_3}$ in $\R^N$ (see Lemma \ref{lem: if i bdd}). But then $v_{i_3}$ is constant, equal to $0$, by Proposition \ref{prop: Liou eq}, and this rules out the possibility that $i_3=1$.

Now, without loss of generality, we suppose that $i_3=3$. We have that on any compact set $K$ of $\R^N$
\[
M_n v_{1,n}(x) v_{2,n}^2(x) v_{3,n}^2(x) \le \,C M_n v_{1,n}^2(x) v_{2,n}^2(x) \frac{v_{3,n}^2(x)}{v_{1,n}(x)}\le C \frac{\sup_K v_{3,n}^2}{\inf_K v_{1,n}} \to 0\,;
\]
hence $\Delta v_{1,n} \to 0$ locally uniformly in $\R^N$, which in turn implies that the function $w_{1,n}$, defined in Lemma \ref{lem: if i unb}, converges to a harmonic limit $w_1$. Since moreover $w_1$ is also globally H\"older continuous and non-constant, this is a contradiction.
\end{proof}

In what follows we focus at ruling out the possibility that exactly one component is unbounded at $0$. To this end, the following preliminary results will be useful.

\begin{lemma}\label{lem: if i unb 2}
Let $\mf{w}_n$ be defined in Lemma \ref{lem: if i unb}. Suppose in addition that $\Delta w_{i,n} \to 0$ in $L^1_{\loc}(\R^N)$ for an index $i \in \{1,2,3\}$. Then $w_{i,n} \to w_i$ strongly in $H^1_{\loc}(\R^N)$, and $w_i$ is harmonic.
\end{lemma}
\begin{proof}
Let $r>0$. It is not difficult to check that $\{w_{i,n}\}$ is bounded in $H^1(B_r)$: indeed, for any $\varphi \in C^\infty_c(B_r)$, we have
\[
\int_{B_r} |\nabla w_{i,n}|^2 \varphi^2\,dx = - 2 \int_{B_r} \varphi w_{i,n} \nabla w_{i,n} \cdot \nabla \varphi\,dx -  \int_{B_r} w_{i,n} \varphi^2 \Delta w_{i,n}\,dx,
\]
whence
\[
\int_{B_r} |\nabla w_{i,n}|^2 \varphi^2\,dx \le C  \int_{B_r}  w_{i,n}^2 |\nabla \varphi|^2\,dx + C \|w_{i,n}\|_{L^\infty({B_r})} \|\Delta w_{i,n}\|_{L^1({B_r})},
\]
which gives the boundedness in $H^1({B_r})$ (recall that the sequence $\{w_{i,n}\}$ is locally bounded in $\R^N$, by definition). Since $r$ was arbitrary, we infer that up to a subsequence $w_{i,n} \weak w_i$ weakly in $H^1_{\loc}(\R^N)$. The strong convergence in $H^1_{\loc}(\R^N)$ can be proved at this point exactly as in Lemma \ref{lem: if i bdd} (we multiply the equation for $w_{i,n}$ by $(w_{i,n}-w_i)\varphi$, and exploit the convergence $w_{i,n} \to w_i$ in $L^\infty_{\loc}(\R^N)$ and the boundedness of $\|\Delta w_{i,n}\|_{L^1({B_r})}$). The fact that the $w_i$ is harmonic follows now by taking the limit in the equation
\[
\int_{\Omega_n} \nabla w_{i,n}\cdot \nabla \varphi\,dx = \int_{\Omega_n} (-\Delta w_{i,n})\varphi\,dx  \quad \forall \varphi \in C^\infty_c(\R^N). \qedhere
\]
\end{proof}

The next lemma is an important consequence of the minimality of $\mf{v}_n$.

\begin{lemma}\label{lem: min one 0}
Suppose that there exist $x_0 \in \R^N$ and $R>\rho>0$ such that the following holds for two indices $i$ and $j$: $v_{j,n} \to 0$ uniformly in $B_R(x_0)$ and in $H^1(B_R(x_0))$, and $v_{i,n} \to v_i$ uniformly in $B_R(x_0)$, with $v_i|_{S_{\rho}(x_0)} \not \equiv  0$. Then $v_i >0$ in $B_\rho(x_0)$.
\end{lemma}

\begin{proof}
Without loss of generality, we fix $x_0=0$, $j=3$, and $i =1$ (in this proof we will not use that $v_1$ is non-constant, hence the same proof works for every possible choice of $i$ and $j$). We suppose by contradiction that $\{v_1=0\} \cap B_\rho \neq \emptyset$, and, exploiting this fact, we will find a contradiction with the minimality of $\mf{v}_n$, Lemma \ref{lem: basic prop}-($v$).

Let
\[
c_n:= J_{M_n}(\mf{v}_n, B_R) = \inf \left\{ J_{M_n}(\mf{u}, B_R) : \mf{u} \in H^1(B_R, \R^3) \text{ such that } \mf{u}-\mf{v}_n \in H_0^1(B_R, \R^3)\right\}.
\]
Let $\eta \in C^1(\overline{B_R})$ be a radial cut-off function with the following properties: $0 \le \eta \le 1$, $\eta$ is radially increasing, $\eta \equiv 0$ in $B_\rho$, $\eta =1$ on $S_{R}$. We define a competitor $\tilde{{v}}_n$ by defining
\[
\tilde v_{1,n} := \begin{cases} v_{1,n} & \text{in $B_R \setminus B_\rho$} \\ \text{harmonic extension of $v_{1,n}$ on $S_\rho$} & \text{in $B_\rho$},
\end{cases} \qquad \tilde v_{2,n}:= v_{2,n}, \quad \tilde v_{3,n}:= \eta v_{3,n}.
\]
By taking the limit as $n \to \infty$, we have that $\tilde v_{1,n} \to \tilde v_1$ in $H^1(B_{\rho}) \cap C(\overline{B_{\rho}})$, where
\[
\tilde v_{1} :=
\text{harmonic extension of $v_{1}$ on $S_\rho$} \text{ in $B_\rho$}.
\]
Since $v_1|_{S_\rho} \not \equiv 0$, by the maximum principle we have that $\tilde v_{1}>0$ in $B_\rho$, and hence $\tilde v_{1} \neq v_{1}$. This implies, by the Dirichlet principle, that
\[
\int_{B_\rho}|\nabla \tilde v_1|^2\,dx < \int_{B_\rho}|\nabla v_1|^2\,dx,
\]
and by $H^1$-convergence we deduce that there exists $\delta>0$ such that
\beq\label{diff 1}
\int_{B_R}\left(|\nabla \tilde v_{1,n}|^2 - |\nabla v_{1,n}|^2\right)\,dx \le  - \delta,
\eeq
for every $n$ large enough. Moreover, since $\tilde v_{3,n} \equiv 0$ in $B_\rho\supset \{\tilde v_{1,n} \neq v_{1,n}\}$,  and $\tilde v_{3,n} \le v_{3,n}$ in $B_R$, we have that
\beq\label{diff prod}
\int_{B_R} M_n \prod_{j=1}^3 \tilde v_{j,n}^2\,dx = \int_{B_R \setminus B_\rho} M_n v_{1,n}^2 v_{2,n}^2 \tilde v_{3,n}^2\,dx \le \int_{B_R} M_n \prod_{j=1}^3 v_{j,n}^2\,dx.
\eeq
Finally, by using the fact that $v_{3,n} \to 0$ in $H^1(B_R(x_0))$, we deduce that
\beq\label{diff 3}
\int_{B_R} \left(|\nabla \tilde v_{3,n}|^2-|\nabla v_{3,n}|^2\right)\,dx \to 0
\eeq
as $n \to \infty$. Therefore, by \eqref{diff 1}-\eqref{diff 3}, we infer that
\[
\begin{split}
c_n &\le J_{M_n}(\tilde{\mf{v}}_n,B_R) \\
& \le J_{M_n}(\mf{v}_n,B_R) + \int_{B_R} \left(|\nabla \tilde v_{1,n}|^2 -|\nabla v_{1,n}|^2\right)\,dx + \int_{B_R} \left(|\nabla \tilde v_{3,n}|^2-|\nabla v_{3,n}|^2\right)\,dx \\
& \le c_n -\delta + o_n(1)
\end{split}
\]
(where $o_n(1) \to 0$ as $n \to \infty$), which is a contradiction for sufficiently large $n$.
\end{proof}

\begin{lemma}\label{lem: 1 a infty 1}
It is not possible that there exists exactly one component $v_{\underline{i},n}(0) \to +\infty$ (with the remaining ones being bounded at $0$), which in addition satisfies
\[
M_n v_{\underline{i},n}^2(0) \le C.
\]
\end{lemma}

\begin{proof}
We argue by contradiction. If $v_{\underline{i},n}(0) \to +\infty$ and $M_n v_{\underline{i},n}^2(0) \le C$, then necessarily $M_n \to 0$, and $\{M_n v_{\underline{i},n}^2\}$ is locally bounded, since $\{v_{\underline{i},n}\}$ has locally bounded oscillation. More precisely
\[
M_n v_{\underline{i},n}^2(x) = M_n v_{\underline{i},n}^2(0)(1+o_n(1)),
\]
locally uniformly in $\R^N$, and hence up to a subsequence $M_n v_{\underline{i},n}^2 \to \lambda$ locally uniformly in $\R^N$, where $\lambda \ge 0$ is a constant. This implies that the other components $\{v_{i_1,n}\}$ and $\{v_{i_2,n}\}$, which are bounded at $0$, and for which the conclusion of Lemma \ref{lem: if i bdd} holds true, are such that $(v_{i_1,n}, v_{i_2,n}) \to (v_{i_1}, v_{i_2})$ locally uniformly and in $H^1_{\loc}(\R^N)$, and the limit functions $v_{i_1}, v_{i_2}$ are globally H\"older continuous in $\R^N$ and satisfy
\[
\begin{cases}
\Delta v_{i_1}= \lambda \, v_{i_1}\, v_{i_2}^2 &  \\
\Delta v_{i_2}= \lambda \, v_{i_1}^2 \, v_{i_2} & \text{in $\R^N$}. \\
v_{i_1}, v_{i,2}  \ge 0
\end{cases}
\]
By Proposition \ref{prop: Liou sys 2} (or by the classical Liouville theorem in case $\lambda=0$), we deduce that $v_{i_1}$ and $v_{i_2}$ are both constant, and hence it is necessary that $\underline{i}=1$ (see Lemma \ref{lem: if i bdd} again).

Now, let $r>0$ be arbitrary. By Lemma \ref{lem: if i bdd}, there exists $C>0$ such that
\[
M_n \int_{B_r}  v_{1,n}^2 \,v_{2,n}^2 \, v_{3,n}\,dx \le C \quad \text{and} \quad \sup_{B_r} v_{3,n} \le C.
\]
Then, using the fact that $\inf_{B_r} v_{1,n} \to +\infty$, we obtain
\[
M_n \int_{B_r}  v_{1,n} \,v_{2,n}^2 \, v_{3,n}^2\,dx \le C M_n \int_{B_r} \frac{v_{1,n}^2}{\inf_{B_r} v_{1,n}} \,v_{2,n}^2 v_{3,n}\,dx \le \frac{C}{\inf_{B_r} v_{1,n}} \to 0,
\]
which in turn gives $\|\Delta v_{1,n}\|_{L^1(B_r)} \to 0$ as $n \to \infty$. If we consider the function $w_{1,n}:= v_{1,n}-v_{1,n}(0)$, since $\Delta w_{1,n}=\Delta v_{1,n}$, Lemmas \ref{lem: if i unb} and \ref{lem: if i unb 2} imply that $w_{1,n} \to w_1$ locally uniformly and strongly in $H^1_{\loc}(\R^N)$, with $w_1$ globally H\"older continuous, harmonic, and non-constant. This is the desired contradiction.
\end{proof}

\begin{lemma}\label{lem: 1 a infty 2}
It is not possible that there exists exactly one component $v_{\underline{i},n}(0) \to +\infty$ (with the remaining ones being bounded at $0$).
\end{lemma}

\begin{proof}
We will split the proof into three steps. We argue by contradiction, and suppose that there exists exactly one component $v_{\underline{i},n}$ such that $v_{\underline{i},n}(0) \to +\infty$. Thanks to Lemma \ref{lem: 1 a infty 1}, up to a subsequence we can suppose in addition that
\[
M_n v_{\underline{i},n}^2(0) \to +\infty,
\]
and, since $v_{\underline{i},n}$ has locally bounded oscillation, we have in fact that
\[
\min_K M_n v_{\underline{i},n}^2 \to +\infty \quad \forall K \subset \R^N \text{ compact}.
\]
Instead the other components $v_{i_1,n}$ and $v_{i_2,n}$ are bounded at $0$, and hence the conclusions of Lemma \ref{lem: if i bdd} hold for them.

\medskip

\emph{Step 1)} One component, say $v_{i_2}$, vanishes identically, and the other satisfies
\[
\Delta v_{i_1}=0 \quad \text{in $\{v_{i_1}>0\}$}.
\]
Moreover,
\beq\label{stima int prod}
\lim_{n \to \infty} M_n \int_{B_r} \prod_{j=1}^3 v_{j,n}^2 \, dx = 0,
\eeq
for every $r>0$.

\medskip

We claim at first that the limits $v_{i_1}$ and $v_{i_2}$ satisfy the segregation condition $v_{i_1}\,v_{i_2} \equiv 0$ in $\R^N$. We take $x_0 \in \R^N$ such that $v_{i_1}(x_0)>0$, and we show that $v_{i_2}(x_0) =0$; since the role of $v_{i_1}$ and $v_{i_2}$ can be exchanged, the claim follows. By continuity and uniform local convergence, if $v_{i_1}(x_0)>0$, then there exist $\rho, \delta>0$ such that
\[
v_{i_1,n}(x) \ge \delta \quad \text{for every $x \in B_{2\rho}(x_0)$, for every $n$ large}.
\]
Therefore, denoting by
\[
I_n:= \inf_{B_{2\rho}(x_0)} M_n v_{\underline{i},n}^2  \to +\infty,
\]
we have that
\[
-\Delta v_{i_2,n} \le -\delta^2 I_n v_{i_2,n}, \quad v_{i_2,n} \ge 0 \quad \text{in $B_{2\rho}(x_0)$},
\]
and moreover $v_{i_2,n}$ is bounded in $B_{2\rho}(x_0)$. The decay estimate in Lemma \ref{lem: decay} gives
\[
\sup_{B_\rho(x_0)} v_{i_2,n} \le Ce^{- C\sqrt{I_n}} \to 0
\]
(the constant $C>0$ depends on $\rho$, but not on $n$), and in particular $v_{i_2}(x_0) = 0$.

This estimate also allows to show that $v_{i_1}$ is harmonic when positive: indeed, $\{v_{i_1}>0\}$ is open by continuity; if $B_{2\rho}(x_0) \subset \{v_{i_1}>0\}$, then
\[
\begin{split}
\sup_{x \in B_{\rho}(x_0)} |\Delta v_{i_1,n}(x)| &= \sup_{x \in B_{\rho}(x_0)} M_n v_{\underline{i},n}^2(x) v_{i_2,n}^2(x) v_{i_1,n}(x) \\
& \le  C I_n(1+o_n(1)) e^{- C\sqrt{I_n}} \to 0
\end{split}
\]
as $n \to \infty$ (we used once again that $v_{\underline{i},n}$ has locally bounded oscillation, so that $\sup_{B_\rho(x_0)} M_n v_{\underline{i},n}^2 =I_n(1+o_n(1))$).

To sum up,
\[
\begin{cases}
\Delta v_{i_1}=0 & \text{in $\{v_{i_1}>0\}$} \\
\Delta v_{i_2}=0 & \text{in $\{v_{i_2}>0\}$} \\
v_{i_1}, v_{i_2} \ge 0, \quad v_{i_1} v_{i_2} \equiv 0 & \text{in $\R^N$}
\end{cases}
\]
and $v_{i_1}, v_{i_2}$ are globally H\"older continuous in $\R^N$. By Proposition \ref{prop: Liou sys 2}, one component, say $v_{i_2}$, vanishes identically.

To complete the proof of Step 1, it remains to prove that \eqref{stima int prod} holds. By the segregation condition, $B_r \subset (B_r \cap \{v_{i_1}=0\}) \cup  (B_r \cap \{v_{i_2}=0\})$. Therefore,
\begin{align*}
M_n \int_{B_r} \prod_{j=1}^3 v_{j,n}^2 \, dx &\le \|v_{i_1,n}\|_{L^\infty(\{v_{i_1}=0\} \cap B_r)} M_n \int_{B_r} v_{i_1,n} \prod_{j \neq i_1} v_{j,n}^2 \, dx \\
& \qquad + \|v_{i_2,n}\|_{L^\infty(\{v_{i_2}=0\} \cap B_r)} M_n \int_{B_r} v_{i_2,n} \prod_{j \neq i_2} v_{j,n}^2 \, dx,
\end{align*}
and \eqref{stima int prod} follows, since the two integrals on the right hand side are bounded by Lemma \ref{lem: if i bdd}.

\medskip

Notice that, since $v_{i_2} \equiv 0$, we have that necessarily $i_2 \neq 1$. Therefore, from now on we assume that $i_2=2$, without loss of generality.

\medskip

\emph{Step 2)} $\underline{i}=3$ and $i_1=1$.

\medskip

Let us consider the functions $w_{\underline{i},n}$ introduced in Lemma \ref{lem: if i unb}. We aim to show that the limit $w_{\underline{i}}$ is constant. By Lemma \ref{lem: if i bdd}, we know that
\[
M_n \int_{B_r} v_{\underline{i},n}^2\, v_{i_2,n}^2\, v_{i_1,n}\,dx \le C
\]
for every $r>0$, since $v_{i_1,n}$ is bounded at $0$. Instead $v_{\underline{i},n} \to +\infty$ locally uniformly, whence
\[
M_n \int_{B_r} v_{\underline{i},n}\, v_{i_2,n}^2\, v_{i_1,n}^2\,dx \le C M_n \int_{B_r} \frac{v_{\underline{i},n}^2}{\inf_{B_r} v_{\underline{i},n}}\, v_{i_2,n}^2\, v_{i_1,n}\,dx \le \frac{C}{\inf_{B_r} v_{\underline{i},n}}\to 0.
\]
We infer that $\Delta w_{\underline{i},n} = \Delta v_{\underline{i},n} \to 0$ in $L^1_{\loc}(\R^N)$, and Lemma \ref{lem: if i unb 2} ensures that $w_{\underline{i}}$ is harmonic in $\R^N$; being also globally H\"older continuous, it must be constant.

\medskip

What we proved so far implies that, under the assumptions of the lemma, the following limits take place locally uniformly: $v_{2,n} \to 0$, $v_{3,n} \to +\infty$ with $w_{3,n} \to w_3 \equiv const.$, and $v_{1,n} \to v_1$ non-constant in $B_1$. Moreover, all the previous limits but the second one are also strong in $H^1_{\loc}(\R^N)$, and \eqref{stima int prod} holds.

\medskip

\emph{Step 3)} Conclusion of the proof.

\medskip

Since $v_1$ is non-constant, and hence non-trivial, in any ball $B_\rho$ with $\rho>1$, while $v_{2,n} \to 0$, the assumptions of Lemma \ref{lem: min one 0} are satisfied. Therefore, $v_1>0$ in $B_\rho$ for every $\rho>1$, namely $v_1>0$ in $\R^N$. It is then a harmonic function in $\R^N$, globally H\"older continuous, and non-constant, a contradiction.
\end{proof}

At this stage, by Lemmas \ref{lem: all infty}, \ref{lem: 2 infty} and \ref{lem: 1 a infty 2}, we deduce that the sequence $\{\mf{v}_n(0)\}$ is bounded. In particular, the conclusions in Lemma \ref{lem: if i bdd} hold for all the components.

\begin{lemma}\label{lem: v e M bdd}
It is not possible that both $\{\mf{v}_n(0)\}$ and $\{M_n\}$ are bounded.
\end{lemma}

\begin{proof}
If $\{\mf{v}_n(0)\}$ and $\{M_n\}$ are both bounded, then up to a subsequence $M_n \to M \ge 0$ and $\mf{v}_n \to \mf{v}$ in $C^{0}_{\loc}(\R^N) \cap H^1_{\loc}(\R^N)$, and the limit function $(v_1,v_2,v_3)$ is globally $\alpha$-H\"older continuous, with $v_1$ non-constant, and satisfy the system
\beq\label{ent sys 308}
\begin{cases}
\Delta v_i = M v_i \prod_{j \neq i} v_j^2 & \text{in $\R^N$}\\
v_i \ge 0 & \text{in $\R^N$}.
\end{cases}
\eeq
If $M=0$, then $v_1$ is harmonic, globally H\"older continuous, and non-constant, a contradiction. We focus then on the case when $M>0$. By the maximum principle, either $v_i>0$ in $\R^N$, or $v_i \equiv 0$, and by Theorem \ref{thm: liou sub} at least one component between $v_2$ and $v_3$ must vanish identically, since $\alpha < \alpha_{3,N}/3$ by assumption ($v_1$ cannot vanish identically, since it is non-constant). Hence, also in this case $v_1$ is harmonic, globally H\"older continuous, and non-constant, a contradiction again.
\end{proof}

It remains to consider the case when $\{\mf{v}_n(0)\}$ is bounded and $M_n \to +\infty$ up to a subsequence.
\begin{lemma}\label{lem: v bdd e M unb}
It is not possible that $\{\mf{v}_n(0)\}$ is bounded and $\{M_n\}$ is unbounded.
\end{lemma}
 \begin{proof}
Let $\{\mf{v}_n(0)\}$ be bounded, and $M_n \to +\infty$. By Lemma \ref{lem: if i bdd}, we deduce that the limits $v_i$ are subharmonic functions such that $v_1\,v_2\,v_3 \equiv 0$ in $\R^N$. Therefore, $(v_1,v_2,v_3)$ satisfies the assumptions of Corollary \ref{cor: liou seg}, whence it follows that one component must vanish identically. We cannot have $v_1 \equiv 0$, since it is non-constant. For concreteness, let $v_3 \equiv 0$ from now on.

Suppose that $v_1(x_0)>0$. We claim that $v_1$ is harmonic in a neighborhood of $x_0$. To prove the claim, let $r>0$.
By local uniform convergence and Lemma \ref{lem: if i bdd}, we infer that
\beq\label{765}
M_n \int_{B_r} v_{1,n}^2 v_{2,n}^2 v_{3,n}^2\,dx \le \|v_{3,n}\|_{L^\infty(B_r)} \int_{B_r} M_n v_{3,n} \prod_{j \neq 3} v_{j,n}^2\,dx \to 0.
\eeq
Now, by continuity and local uniform convergence, there exists $\rho, \delta>0$ such that
\[
\inf_{B_{\rho}(x_0)} v_{1,n} \ge \delta \quad \text{for $n$ sufficiently large}.
\]
Therefore, using also \eqref{765}, we obtain
\[
\|\Delta v_{1,n}\|_{L^1(B_{\rho}(x_0))} = M_n\int_{B_{\rho}(x_0)} v_{1,n} \, v_{2,n}^2 \, v_{3,n}^2\,dx  \le \frac1{\delta} M_n\int_{B_{\rho}(x_0)} v_{1,n}^2 \, v_{2,n}^2 \, v_{3,n}^2\,dx \to 0,
\]
and recalling that $v_{1,n} \to v_1$ in $H^1_{\loc}(\R^N)$, we deduce that $v_1$ is harmonic in $B_\rho(x_0)$. This argument shows that $v_1$ is harmonic in its positivity set. The same holds for $v_2$.

We are ready to reach a contradiction. Since $v_1$ is non-constant, and hence non-trivial, in any ball $B_\rho$ with $\rho>1$, while $v_{3,n} \to 0$, the assumptions of Lemma \ref{lem: min one 0} are satisfied (notice that, if $v_1$ is non-trivial in $B_\rho$, then $v_1|_{S_\rho} \not \equiv 0$ by subharmonicity). Therefore, $v_1>0$ in $B_\rho$ for every $\rho>1$, namely $v_1>0$ in $\R^N$. It is then a harmonic function in $\R^N$, globally H\"older continuous, and non-constant, a contradiction.
\end{proof}

\begin{proof}[Conclusion of the proof of Theorem \ref{thm: holder bounds}]
Lemmas \ref{lem: basic prop}-\ref{lem: v bdd e M unb} imply that $\{\mf{u}_\beta\}$ is bounded in $C^{0,\alpha}_{\loc}(\Omega)$, as claimed. This is true for every $\alpha \in (0,\alpha_{3,N}/3)$. In turn, by the Arzel\`a--Ascoli theorem, the local $C^{0,\alpha}$ convergence to some limit $\tilde{\mf{u}}$ follows, up to a subsequence.

The other properties can be proved as in Lemmas \ref{lem: basic prop} and \ref{lem: v bdd e M unb}. We only give a sketch. By testing the equation for $u_{i,\beta}$ with a cutoff function $\varphi \in C^\infty_c(\R^N)$, we obtain
\[
\beta \int_{\Omega} \Big(\prod_{j \neq i} u_{j,\beta}^2 \Big) u_{i,\beta} \varphi \,dx = \int_{\Omega} u_{i,\beta} \Delta \varphi \,dx \le C.
\]
Thus, by local uniform convergence as $\beta \to +\infty$, the limit $\tilde{\mf{u}}$ satisfies the partial segregation condition $u_1\,u_2\,u_3 \equiv 0$ in $\Omega$.

In order to show the strong $H^1_{\loc}$ convergence, we test at first the equation for $u_{i,\beta}$ with $u_{i,\beta} \varphi^2$, where $\varphi \in C^\infty_c(\Omega)$ is an arbitrary nonnegative cutoff function, and integrate by parts. As in Lemma \ref{lem: basic prop}, we obtain boundedness in $H^1(\omega)$ for every $\omega \ssubset \Omega$, whence weak $H^1_{\loc}$-convergence follows. The strong convergence can be proved by testing the equation for $u_{i,\beta}$ with $(u_{i,\beta}-u_i) \varphi$, and integrating by parts (see Lemma \ref{lem: basic prop} again).

At this point, the fact that as $\beta \to +\infty$
\[
\beta \int_{\omega} \prod_{j=1}^3 u_{j,\beta}^2\,dx \to 0 \quad \text{for every $\omega \ssubset \Omega$}
\]
can be proved exactly as in Lemma \ref{lem: v bdd e M unb}.
\end{proof}

\begin{remark}
For future convenience, we point out that the limitation $\alpha < \bar \nu= \alpha_{3,N}/3$ in Theorem \ref{thm: holder bounds} plays a role only in Lemmas \ref{lem: v e M bdd} and \ref{lem: v bdd e M unb}, where we apply Theorem \ref{thm: liou sub} and Corollary \ref{cor: liou seg}. In particular, we can increase the threshold $\bar \nu$ in Theorem \ref{thm: holder bounds} by establishing improved Liouville-type theorems for globally H\"older continuous solutions to \eqref{ent sys 308}, minimal with respect to variations of the associated energy with compact support (it is plain that, if $\{M_n\}$ is bounded, so that $M_n \to M \ge 0$ up to a subsequence, the limiting profile $\mf{v}$ is itself minimal for the associated energy) and improved Liouville-type theorems for globally H\"older continuous functions in $\mathcal{L}_{\loc}(\R^N)$. This observation will be used in the proof of optimal regularity, see \cite{ST2024}.
\end{remark}

\subsection{Scale-invariant version of the H\"older estimates}\label{sub: scale-invariant}

We conclude this section by proving the scale-invariant variant of
Theorem \ref{thm: holder bounds} announced in Remark
\ref{rem:scale-invariant}. The point is that assumption (h1) enters the
proof of Theorem \ref{thm: holder bounds} in exactly two places, and in
both of them only through the ratio between the $L^\infty$ norm of the
solution and the diverging H\"older seminorm $L_n$. It is therefore enough
to build that ratio into the contradiction argument from the very
beginning, leaving the rest of the proof untouched.

\begin{proof}[Proof of the scale-invariant estimate - see Remark \ref{rem:scale-invariant}]
By a standard covering and scaling argument we may assume that
$\Omega \supset B_3$, $\Omega' = B_1$ and $\Omega'' = B_3$; we recall
that the dilation $x \mapsto x_0 + r x$ maps solutions of
\eqref{P beta} with parameter $\beta$ into solutions with parameter
$\beta r^2$, preserving (h2). Let $\eta$ be the cut-off function fixed
at the beginning of this section, and fix
$\alpha \in (0, \alpha_{3,N}/3)$.

Assume by contradiction that \eqref{eq:scale-invariant} fails: then
there exist $\beta_n > 0$ and solutions $\mathbf{u}_n$ of
\eqref{P beta} with $\beta = \beta_n$, satisfying (h2), such that,
setting
\[
  \mathcal{N}_n := \|\mathbf{u}_n\|_{L^\infty(B_3)},
  \qquad
  L_n := \sup_{i=1,2,3} \ \sup_{\substack{x \neq y \\ x,y \in B_3}}
  \frac{|(\eta u_{i,n})(x) - (\eta u_{i,n})(y)|}{|x-y|^{\alpha}},
\]
we have
\begin{equation}\label{eq:new-contradiction}
  L_n \;>\; n \, \mathcal{N}_n,
  \qquad\text{that is}\qquad
  \frac{\mathcal{N}_n}{L_n} \;\longrightarrow\; 0
  \quad\text{as } n \to +\infty.
\end{equation}
(We may clearly assume $0 < \mathcal{N}_n < +\infty$, since otherwise
there is nothing to prove; in particular $L_n > 0$.) As in the proof of
Theorem \ref{thm: holder bounds}, up to a relabelling the supremum
defining $L_n$ is achieved for $i = 1$ and at a pair of points
$x_n, y_n \in \overline{B_2}$, with $x_n \neq y_n$, and we set
$r_n := |x_n - y_n|$. We then define the blow-up sequences
$\mathbf{v}_n$ and $\bar{\mathbf{v}}_n$ exactly as in Subsection
\ref{subsec:blowup}, with the same $L_n$.

Assumption (h1) was used only twice, and both occurrences are now
replaced by \eqref{eq:new-contradiction}:

\smallskip
\noindent \emph{(a) The blow-up radii still vanish.} Since
$0 \le \eta \le 1$ and $\eta \equiv 0$ outside $B_2$, we have
$|\eta u_{1,n}| \le \mathcal{N}_n$, whence
\[
  r_n^{\alpha}
  = \frac{|(\eta u_{1,n})(x_n) - (\eta u_{1,n})(y_n)|}{L_n}
  \;\le\; \frac{2\,\mathcal{N}_n}{L_n} \;\longrightarrow\; 0,
\]
so that $r_n \to 0^+$ and the scaled domains $\Omega_n$ exhaust
$\mathbb{R}^N$.

\smallskip
\noindent \emph{(b) Lemma \ref{lem: basic prop}-(iii).} Since $\eta$ is
globally Lipschitz with constant $l$, for every $x$ in a compact set $K$
and every $n$ large enough (so that $x_n + r_n x \in B_3$) we obtain
\[
  |v_{i,n}(x) - \bar{v}_{i,n}(x)|
  = \frac{|u_{i,n}(x_n + r_n x)|}{L_n r_n^{\alpha}}
    \,\bigl| \eta(x_n) - \eta(x_n + r_n x) \bigr|
  \;\le\; \frac{l\, \mathcal{N}_n \, r_n^{1-\alpha}}{L_n}\,|x|
  \;\longrightarrow\; 0
\]
uniformly on $K$, because $\alpha < 1$, $r_n \to 0$ and, by
\eqref{eq:new-contradiction}, $\mathcal{N}_n / L_n \to 0$.

\smallskip
Every remaining step of Section \ref{sec: hold bounds} involves only the
rescaled functions $\mathbf{v}_n$, the rescaled coupling constant
$M_n = \beta_n r_n^{2+4\alpha} (L_n/\eta(x_n))^4$, and the normalization
of the H\"older seminorm of $\bar{\mathbf{v}}_n$, and is therefore
unchanged. In particular, Lemmas \ref{lem: basic prop}--%
\ref{lem: v bdd e M unb} rule out every possible behavior of the pair
$(\mathbf{v}_n(0), M_n)$, which is the desired contradiction.
\end{proof}

\begin{remark}\label{rem:beta-large}
It is worth observing that the proof of Theorem
\ref{thm: holder bounds} never uses the assumption
$\beta_n \to +\infty$ (nor $\beta > 1$) directly: the parameter
$\beta_n$ only enters through the rescaled coupling constant $M_n$,
whose possible behaviors --- bounded or unbounded --- are both
discussed in Lemmas \ref{lem: v e M bdd} and
\ref{lem: v bdd e M unb}. This is why Remark
\ref{rem:scale-invariant} may be stated for every $\beta > 0$.
\end{remark}

\section{Proof of Theorem \ref{thm: min fixed traces}}\label{sec: corol fixed}

In the first part of this section, we focus on the uniform H\"older estimates up to the boundary. This is the most delicate point in Theorem \ref{thm: min fixed traces}. The rest of the proof is rather standard, and will be the content of Subsection \ref{sub: concl}

\subsection{Uniform H\"older estimates up to the boundary}\label{sub: boundary}

The strategy is analogous to that used in the proof of Theorem \ref{thm: holder bounds}. By minimality, each $\mf{u}_\beta$ solves system \eqref{PSP} for the corresponding value of $\beta$. In particular, $u_{i,\beta}$ is subharmonic, and hence
\[
\sup_\beta \, \sup_{i=1,2,3} u_{i,\beta} \le \|\boldsymbol{\psi}\|_{L^\infty(\pa \Omega)}<+\infty,
\]
namely $\{\mf{u}_\beta\}$ satisfies assumption (h1) of Theorem \ref{thm: holder bounds}. Plainly, $\{\mf{u}_\beta\}$ also satisfies assumption (h2).

Now, for $\alpha \in (0,\alpha_{3,N}/3)$ fixed, we aim to prove that the family $\{\mf{u}_\beta\}_{\beta >1}$ admits a uniform bound on the $\alpha$-H\"older semi-norm, that is, there exists $C > 0$, independent of $\beta$, such that
\[
\sup_{i=1,2,3} \sup_{\substack{x \neq y \\ x,y \in \overline{\Omega}}}    \frac{ |u_{i,\beta}(x)-u_{i,\beta}(y)|}{|x-y|^{\alpha}} \leq C.
\]

Let us assume by contradiction that there exists a sequence $\beta_n \to +\infty$ and a corresponding sequence $\{\mf{u}_n:= \mf{u}_{\beta_n}\}$ such that
\[
    L_n := \sup_{i = 1, 2, 3} \sup_{\substack{x \neq y \\ x,y \in \overline{\Omega}}} \frac{ |u_{i,n}(x)- u_{i,n}(y)|}{|x-y|^{\alpha}} \to \infty \qquad \text{as $n \to +\infty$.}
\]
As in Theorem \ref{thm: holder bounds}, up to a relabelling, the supremum is achieved for $i = 1$ and at a pair of points $x_n, y_n \in \overline{\Omega}$, with $x_n \neq y_n$ and $|x_n-y_n| \to 0$ as $n \to \infty$. We consider the blow-up sequence
   \[
    v_{i,n}(x) := \frac{u_{i,n}(x_n + r_n x)}{L_n r_n^{\alpha}}, \quad \text{where} \quad r_n := |x_n -y_n| \to 0^+, \quad \text{and} \quad x \in \frac{\Omega - x_n}{r_n} =: \Omega_n,
\]
and we have two different possibilities according to the fact that $\dist(x_n,\pa \Omega)/r_n$ is bounded or not. In the latter case, the scaled domains $\Omega_n$ exhaust $\R^N$ as $n \to \infty$, and the contradiction is reached proceeding exactly as in the proof of Theorem \ref{thm: holder bounds}. If instead $\dist(x_n,\pa \Omega)/r_n$ remains bounded, then $\Omega_n$ tends to a half-space $\Omega_\infty$ and the previous proof has to be conveniently modified. At first, we observe that properties ($i$) and ($v$) of Lemma \ref{lem: basic prop} still hold in the present setting. Regarding property ($ii$), we have that
\[
\begin{cases}
\Delta v_{i,n}=M_n v_{i,n} \prod_{j \neq i} v_{j,n}^2, \quad v_{i,n}>0 & \text{in $\Omega_n$}\\
v_{i,n} = \psi_{i,n} & \text{on $\pa \Omega_n$},
\end{cases}
\]
where
\[
 M_n:= \beta_n r_n^{2+4\alpha} L_n^{4}, \quad \text{and} \quad \psi_{i,n}(x) := \frac{\psi_i(x_n+r_nx)}{r_n^\alpha L_n}.
\]
Since $\boldsymbol{\psi}_n$ is a H\"older scaling of fixed Lipschitz functions, and since $\{\mf{v}_n\}$ has bounded $\alpha$-H\"older semi-norm, we can derive important information about the limiting behavior at the boundary of the half-space. We refer to \cite[Lemma 4.3]{ST23} for the proof.

\begin{lemma}\label{lem: on boundary}
It is possible to extend $\mf{v}_n$ outside $\Omega_n$ in a Lipschitz fashion, in such a way that:
\begin{itemize}
\item[($i$)] If $\{v_{i,n}(0)\}$ is bounded, then
$v_{i,n} \to v_i$ in $C^{0,\alpha'}_{\loc}(\R^N)$ for every $0<\alpha'<\alpha$, up to a subsequence; moreover, the limit function $v_i$ attains a constant value on the boundary $\pa \Omega_\infty$.
\item[($ii$)] If $\{v_{i,n}(0)\}$ is unbounded, then $w_{i,n}(x) := v_{i,n}(x)-v_{i,n}(0)$ converges to $w_i$ in $C^{0,\alpha'}_{\loc}(\R^N)$ for every $0<\alpha'<\alpha$, up to a subsequence; moreover, the limit function $w_i$ attains a constant value on the boundary $\pa \Omega_\infty$, and $w_1$ is non-constant.
\end{itemize}
\end{lemma}

\begin{remark}\label{rem: if i bdd}
Furthermore, if $\{v_{i,n}(0)\}$ is bounded, then the thesis of Lemma \ref{lem: if i bdd} holds, with the only differences that in point ($iii$) we have to require $\omega \ssubset \Omega_\infty$, and in ($iv$) we have convergence in $H^1_{\loc}(\Omega_\infty)$.
\end{remark}

At this point the idea is to proceed as in the proof of Theorem \ref{thm: holder bounds}, discussing all the possible cases (all the components are unbounded at $0$, exactly $2$ components are unbounded at $0$, etc.), and obtaining a contradiction in all of them. In this perspective, it is important to observe that, thanks to the partial segregation condition and to the uniform boundedness of the $\alpha$-H\"older seminorm, for every $R>1$ such that $\pa \Omega_\infty \cap B_R \neq \emptyset$ at least one component of $\mf{v}_n$, say $v_{i_n,n}$, must vanish at some point of $\pa \Omega_{n} \cap B_R$. Since $i_n$ varies in the set of indices $\{1,2,3\}$, which is finite, up to a subsequence we can suppose that $i_n = \underline{i}$ for every $n$. Thus, Lemma \ref{lem: on boundary} ensures that $v_{\underline{i}}$, being constant on $\pa \Omega_\infty$, vanishes on the whole $\pa \Omega_\infty$. Moreover, recalling that $\{\mf{v}_n\}$ has locally bounded oscillation, $\{v_{\underline{i},n}\}$ is locally bounded in $\Omega_\infty$. We can actually prove a much stronger statement:

\begin{lemma}\label{lem: 0 half-space}
We have that $v_{\underline{i}} \equiv 0$ in $\Omega_{\infty}$.
\end{lemma}

\begin{proof}
Without loss, suppose that $\Omega_\infty= \R^N_+$. We know that $v:= v_{\underline{i}}$ is subharmonic in $\R^N_+$, $\alpha$-H\"older continuous up to the boundary, with $v=0$ on $\pa \R^N_+$. We extend $v$ as the trivial function outside of $\R^N_+$, and we claim that such extension, still denoted by $v$, is subharmonic on $\R^N$. Indeed, recall that a continuous function is subharmonic in an open set if and
only if it satisfies the sub-mean value inequality on all sufficiently small
balls centered at each point of that set. Let then $x \in \mathbb{R}^N$. If
$x \in \mathbb{R}^N_+$, then $B_r(x) \subset \mathbb{R}^N_+$ for every
$0 < r < x_N$, whence
\[
  v(x) \le \frac{1}{|B_r(x)|} \int_{B_r(x)} v \, dy
\]
by subharmonicity of $v$ in $\mathbb{R}^N_+$. If instead
$x \notin \mathbb{R}^N_+$, then $v(x) = 0$, either because $v$ has been
extended trivially outside $\mathbb{R}^N_+$, or because $v$ vanishes on
$\partial \mathbb{R}^N_+$, and the same inequality holds for every $r > 0$,
simply because $v \ge 0$ in the whole of $\mathbb{R}^N$. This proves the claim.
Now, since $v$ is subharmonic and continuous in $\R^N$, it is in $H^1_{\loc}(\R^N)$ (see e.g. \cite[Exercise 2.3]{PSU12}).

Let us define $v'$ as the even reflection of $v$ across $\pa \R^N_+$. Clearly, $v$ and $v'$ are subharmonic continuous function in $H^1_{\loc}(\R^N)$, with disjoint positivity sets, globally H\"older continuous. Therefore, by Proposition \ref{prop: Liou sys 2}, one of them vanishes identically. By definition of $v'$, we infer that $v= v_{\underline{i}} \equiv 0$, as desired.
\end{proof}

The previous lemma is extremely useful: it ensures that, when $\Omega_n$ tends to a half-space, one component is not only bounded, but vanishes identically in the limit. From now on, for the sake of concreteness, we suppose that $\underline{i}=3$, so that $v_{3,n} \to v_3 \equiv 0$ locally uniformly (notice that it cannot be $\underline{i}=1$, since $v_{1,n}$ has constant oscillation equal to $1$ in $\overline{B_1} \cap \overline{\Omega_\infty}$).

Now, in order to show the validity of estimate \eqref{global bounds}, we have to reach a contradiction in the following cases:
\begin{itemize}
\item[($i$)] exactly two components are unbounded at $0$;
\item[($ii$)] exactly one component is unbounded at $0$;
\item[($iii$)] $\{\mf{v}_n(0)\}$ is bounded.
\end{itemize}
We will frequently use the following elementary result (see \cite[Lemma 4.2]{ST23}).

\begin{lemma}\label{lem: liou half}
Let $H$ be a half-space, and suppose that $v \in C^0(\overline{H}) \cap C^2(H)$ is a harmonic function in $H$, with $v|_{\pa H} = const.$, and $v$ globally $\alpha$-H\"older continuous for some $\alpha \in (0,1)$. Then $v$ is constant in $H$.
\end{lemma}

We can proceed with the core of the proof of Estimate \eqref{global bounds}.

\begin{lemma}\label{lem: 2 a infty global}
Suppose that up to a subsequence $v_{1,n}(0), v_{2,n}(0) \to +\infty$. Then
\[
M_n v_{1,n}^2(0) \, v_{2,n}^2(0) \le C
\]
for a positive constant $C>0$ independent of $n$.
\end{lemma}

\begin{proof}
The proof is similar to the one of Lemma \ref{lem: 2 a infty}. We argue by contradiction, and suppose that up to a subsequence $M_n v_{1,n}^2(0) v_{2,n}^2(0) \to +\infty$ as $n \to \infty$.

Since $\{\mf{v}_n\}$ has local bounded oscillation, for every $x_0 \in \Omega_\infty$ and $R>1$ we have that
\[
\inf_{\pa \Omega_n \cap B_R(x_0)}  v_{1,n}, \ \inf_{\pa \Omega_n \cap B_R(x_0)} v_{2,n} \to +\infty.
\]
By partial segregation of the boundary datum, this implies that $v_{3,n}|_{\pa \Omega_n \cap B_R(x_0)} =0$, so that we can extend $v_{3,n}$ as the $0$ function on $B_R(x_0) \setminus \Omega_n$, obtaining a continuous subsolution to
\[
-\Delta v_{3,n} \le - I_n v_{3,n}
\quad \text{in the whole $B_{R}(x_0)$},
\]
where
\[
I_n := M_n \inf_{B_{R}(x_0) \cap \Omega_n} v_{1,n}^2 \, v_{ 2,n}^2 \ge M_n \left(v_{1,n}(0)- C_R\right)^2 \left(v_{2,n}(0)- C_R\right)^2 \to +\infty.
\]
By Lemma \ref{lem: decay}, and recalling that $v_{3,n} \to 0$ and $I_n \to +\infty$, we infer that
\[
\sup_{B_{R/2}(x_0) \cap \Omega_n} v_{3,n} \le C \sup_{B_{R}(x_0) \cap \Omega_n} v_{3,n} e^{- \frac{R\sqrt{I_n}}{2}} \to 0.
\]
As in Lemma \ref{lem: 2 a infty}, this estimate allows to show that
\[
\begin{split}
|\Delta v_{1,n}(x)| & \le  M_n v_{1,n}^2(x) \prod_{j \neq 1} v_{j,n}^2(x)  \le M_n \big( v_{1,n} v_{2,n}(x)\big)^2 C_R e^{-C_R\sqrt{I_n}}   \le  C_R  I_n  e^{-C_R\sqrt{I_n}}
\end{split}
\]
for every $x \in B_R(x_0) \cap \Omega_n$, for every $R>1$ and $x_0 \in \Omega_\infty$. In particular, this holds true in a neighborhood of $x_0$, where we recall that $x_0 \in \Omega_\infty$ was arbitrary. This means that, if we consider $w_{1,n}:= v_{1,n}-v_{1,n}(0)$, then $w_{1,n} \to w_1$ locally uniformly in $\Omega_\infty$, with $w_1$ globally H\"older continuous in $\Omega_\infty$, non-constant, constant on the boundary (recall Lemma \ref{lem: on boundary}), and harmonic. Indeed, $|\Delta w_{1,n}| = |\Delta v_{1,n}| \to 0$ locally uniformly in $\Omega_\infty$. By Lemma \ref{lem: liou half}, this is the desired contradiction.
\end{proof}

\begin{lemma}\label{lem: 2 infty global}
It is not possible that $v_{1,n}(0), v_{2,n}(0) \to +\infty$.
\end{lemma}

\begin{proof}
In case $v_{1,n}(0), v_{2,n}(0) \to +\infty$, we know by Lemma \ref{lem: 2 a infty global} that
\[
M_n v_{1,n}^2(0) v_{2,n}^2(0) \le C.
\]
Since $v_{j,n}/v_{j,n}(0) \to 1$ locally uniformly, for $j=1, 2$, it follows that
\[
|M_n v_{1,n}^2 v_{2,n}^2 - M_n v_{1,n}^2(0) v_{2,n}^2(0)| \to 0
\]
uniformly on any compact set of $\Omega_\infty$. Therefore, $M_n v_{1,n}^2 v_{2,n}^2$ is in turn locally bounded, and recalling that $v_{3,n} \to v_3 \equiv 0$ locally uniformly, we deduce that
\[
M_n v_{1,n}(x) v_{2,n}^2(x) v_{3,n}^2(x) \le M_n v_{1,n}^2(x) v_{2,n}^2(x) v_{3,n}^2(x)\le C v_{3,n}^2(x) \to 0;
\]
hence $\Delta v_{1,n} \to 0$ locally uniformly in $\Omega_\infty$, which in turn implies that the function $w_{1,n} =v_{1,n}-v_{1,n}(0)$ converges to a harmonic limit $w_1$. Since moreover $w_1$ is also globally H\"older continuous and non-constant, but constant on the boundary of $\pa \Omega_\infty$, Lemma \ref{lem: liou half} gives a contradiction.
\end{proof}

\begin{lemma}\label{lem: 1 a infty 1 glob}
It is not possible that there exists exactly one component $v_{j,n}(0) \to +\infty$ (with the remaining ones being bounded at $0$), which in addition satisfies
\[
M_n v_{j,n}^2(0) \le C.
\]
\end{lemma}

\begin{proof}
We denote by $j \in \{1,2\}$ the index such that $v_{j,n}(0) \to +\infty$ up to a subsequence, and $i$ the other. Since $v_{j,n}/v_{j,n}(0) \to 1$ locally uniformly in $\Omega_\infty$, we have that $\{M_n v_{j,n}^2\}$ is locally bounded as well, and in fact
\[
|M_n v_{j,n}^2 - M_n v_{j,n}^2(0) | \to 0
\]
uniformly on any compact set of $\Omega_\infty$. Therefore, by taking the limit in the equation for $v_{i,n}$, we deduce that there exists a constant $\lambda \ge 0$ such that
\[
\Delta v_{i} = \lambda v_3^2 v_{i} =0 \quad \text{in $\Omega_\infty$}.
\]
Since we also know by Lemma \ref{lem: on boundary} that $v_{i}$ is constant on $\pa \Omega_\infty$, and globally H\"older continuous, Lemma \ref{lem: liou half} implies that $v_{i}$ is constant, and hence $i=2$ and $j=1$. At this point, exactly as in the conclusion of the proof of Lemma \ref{lem: 1 a infty 1}, we can prove that $\Delta v_{1,n} \to 0$ in $L^1_{\loc}(\Omega_\infty)$, and Lemma \ref{lem: if i unb 2} ensures that $w_{1,n} =v_{1,n}-v_{1,n}(0)$ converges to a harmonic limit $w_1$. This gives a contradiction as usual.
\end{proof}

\begin{lemma}
It is not possible that there exists exactly one component $v_{j,n}(0) \to +\infty$ (with the remaining ones being bounded at $0$).
\end{lemma}

\begin{proof}
We denote by $j \in \{1,2\}$ the index such that $v_{j,n}(0) \to +\infty$ up to a subsequence, and $i$ the other. By Lemma \ref{lem: 1 a infty 1 glob}, we have that $M_n v_{j,n}^2(0) \to +\infty$. We claim that $v_{i}$ is harmonic when positive. To prove the claim, notice that, by continuity and uniform local convergence, if $x_0 \in \Omega_\infty$ and $v_{i}(x_0)>0$, then there exist $\rho, \delta>0$ such that
\[
v_{i,n}(x) \ge \delta \quad \text{for every $x \in B_{2\rho}(x_0) \subset \Omega_\infty$, for every $n$ large}.
\]
Therefore, denoting by
\[
I_n:= \inf_{B_{2\rho}(x_0)} M_n v_{j,n}^2  \to +\infty,
\]
we have that
\[
-\Delta v_{3,n} \le -\delta^2 I_n v_{3,n}, \quad v_{3,n} \ge 0 \quad \text{in $B_{2\rho}(x_0)$},
\]
and moreover $v_{3,n}$ is bounded in $B_{2\rho}(x_0)$. The decay estimate in Lemma \ref{lem: decay} gives
\[
\sup_{B_\rho(x_0)} v_{3,n} \le Ce^{- C\sqrt{I_n}}.
\]
Therefore,
\[
\begin{split}
\sup_{x \in B_{\rho}(x_0)} |\Delta v_{i,n}(x)| &= \sup_{x \in B_{\rho}(x_0)} M_n v_{j,n}^2(x) v_{3,n}^2(x) v_{i,n}(x) \\
& \le  C I_n(1+o_n(1)) e^{- C\sqrt{I_n}} \to 0
\end{split}
\]
as $n \to \infty$ (we used once again that $v_{j,n}$ has locally bounded oscillation), and the claim follows.

Now, $\Delta v_{i}=0$ in $\{v_{i}>0\}$, $v_{i}$ is constant on $\pa \Omega_\infty$, and $v_{i}$ is globally H\"older continuous in $\overline{\Omega_\infty}$. If $v_{i}=0$ on $\pa \Omega_\infty$, then as in Lemma \ref{lem: 0 half-space} we conclude that $v_{i} \equiv 0$ in $\Omega_\infty$, and necessarily $i=2$ and $j=1$. If instead $v_{i} = c >0$ on $\pa \Omega_\infty$, then the global H\"older continuity implies that $v_{i} \ge c/2$ in a tubular neighborhood of $\pa \Omega_\infty$. Using this fact and Lemma \ref{lem: min one 0} (which holds true provided that $B_R(x_0) \subset \Omega_\infty$), it is not difficult to deduce that $v_{i}>0$ in the whole half-space $\Omega_\infty$. As a result, $v_{i}$ must be constant, and also in this case $i=2$ and $j=1$.

To sum up, we have that $v_{1,n} \to +\infty$ locally uniformly, while $v_{2,n} \to const.$ and $v_{3,n} \to 0$ locally uniformly. Recall also that, thanks to the boundedness of $\{v_{3,n}(0)\}$, for every $\omega \ssubset \Omega_\infty$ we have
\[
M_n \int_{\omega} v_{1,n}^2\,v_{2,n}^2 v_{3,n}\,dx \le C
\]
(see Remark \ref{rem: if i bdd}). Therefore,
\[
M_n \int_{\omega} v_{1,n}^2\,v_{2,n}^2 v_{3,n}^2\,dx \to 0,
\]
whence it follows that
\[
M_n \int_{\omega} v_{1,n}\,v_{2,n}^2 v_{3,n}^2\,dx \le M_n \int_{\omega} \frac{v_{1,n}^2}{\inf_\omega v_{1,n}} \,v_{2,n}^2 v_{3,n}^2\,dx \to 0.
\]
Hence $\Delta v_{1,n} \to 0$ locally uniformly in $\Omega_\infty$, the function $w_{1,n} =v_{1,n}-v_{1,n}(0)$ converges to a harmonic limit $w_1$, and $w_1$ is also globally H\"older continuous and non-constant, but constant on $\pa \Omega_\infty$, in contradiction with Lemma \ref{lem: liou half}.
\end{proof}

\begin{proof}[Conclusion of the proof of the uniform H\"older estimate of $\{\mf{u}_\beta\}$]
It remains to show that it is not possible that $\{\mf{v}_n(0)\}$ is bounded. Suppose that this is the case, and suppose moreover that $\{M_n\}$ is bounded. Then, up to a subsequence, $\mf{v}_n \to \mf{v}$ locally uniformly in $\Omega_\infty$, $M_n \to M \ge 0$, and $\mf{v}$ solves
\[
\Delta v_i = M v_i \prod_{j \neq i} v_j^2 \quad \text{in $\Omega_\infty$}.
\]
We also know that $v_3 \equiv 0$, and hence $v_1$ is harmonic in a half-space, globally H\"older continuous, non-constant, with constant boundary datum, a contradiction.
Suppose now that $M_n \to +\infty$. As in Lemma \ref{lem: v bdd e M unb}, we show that $v_1$ is harmonic when positive, and, being non-constant (and hence non-trivial), Lemma \ref{lem: min one 0} ensures that $v_1>0$ in $\Omega_\infty$. Thus, we reach the same contradiction as before, and the proof of the validity of the uniform H\"older estimate up to the boundary is complete.
\end{proof}

\subsection{Conclusion of the proof of Theorem \ref{thm: min fixed traces}}\label{sub: concl} Theorem \ref{thm: holder bounds} and the uniform H\"older estimate up to the boundary imply that up to a subsequence $\mf{u}_\beta \to \tilde{\mf{u}}$ in $H^1_{\loc}(\Omega)$ and in $C^{0,\alpha}(\overline{\Omega})$, for every $\alpha \in (0,\bar \nu)$. The only things to prove are that the limit is minimal for the problem, and that the $H^1$ convergence is global in $\Omega$. Let us set
\[
c_\beta:= \inf_{\mf{u} \in H_{\boldsymbol{\psi}}} J_\beta(\mf{u}, \Omega),
\]
and recall the definition of $c_\infty$ given in \eqref{def c infty}.
We observe that, if $\mf{u}$ satisfies the partial segregation condition, then $J_\beta(\mf{u}, \Omega)$ is independent of $\beta$, and is equal to the Dirichlet energy of $\mf{u}$. Therefore $c_\beta \le c_\infty$. Then, by the minimality of $\mf{u}_\beta$, for every $\beta>1$ we have $J_\beta(\mf{u}_\beta) \le c_\infty$. Since moreover $u_{i,\beta}\equiv \psi_i$ on $\pa \Omega$, the uniform $H^1(\Omega,\R^3)$ boundedness of $\{\mf{u}_\beta\}$ follows. Hence, up to a subsequence, $\mf{u}_\beta \rightharpoonup \tilde{\mf{u}}$ weakly in $H^1(\Omega,\R^3)$ and a.e. in $\Omega$, and we recall that $\tilde{\mf{u}}$ satisfies the partial segregation condition.

Now, by the minimality of $\mf{u}_\beta$ and weak convergence,
\begin{align*}
c_\infty & \le \sum_{i=1}^3 \int_{\Omega} |\nabla \tilde u_{i}|^2\,dx \le \liminf_{\beta \to \infty} \sum_{i=1}^3 \int_{\Omega} |\nabla u_{i,\beta}|^2\,dx\\
&  \le \limsup_{\beta \to \infty} J_\beta(\mf{u}_\beta,\Omega) =\limsup_{\beta \to \infty} c_\beta\le c_\infty.
\end{align*}
This means that all the previous inequalities are indeed equalities, and in particular:
\begin{itemize}
\item $\|\nabla u_{i,\beta}\|_{L^2(\Omega)} \to \|\nabla \tilde u_i\|_{L^2(\Omega)}$, which together with the weak convergence and the fact that $\mf{u}_\beta = \boldsymbol{\psi}$ on $\partial \Omega$, for every $\beta$, ensures that $\mf{u}_\beta \to \tilde{\mf{u}}$ strongly in $H^1(\Omega,\R^{3})$ (recall that $\Omega$ is bounded);
\item we have that
\[
 \lim_{\beta \to \infty} \sum_{i=1}^3 \int_{\Omega} |\nabla u_{i,\beta}|^2 \\
 = \lim_{\beta \to \infty} J_\beta(\mf{u}_\beta,\Omega),
\]
so that
\[
\lim_{\beta \to +\infty} \beta \int_{\Omega} \prod_{j=1}^3 u_{j,\beta}^2\,dx = 0.
\]
\item we have that
\[
c_\infty= \sum_{i=1}^3 \int_{\Omega} |\nabla \tilde u_{i}|^2\,dx,
\]
which proves the minimality of $\tilde{\mf{u}}$.
\end{itemize}

\section{Proof of Theorem \ref{thm: bas prop}}\label{sec: prop}

We start with a preliminary result. Let $\Omega \subset \R^N$ be a domain, $M>0$, and let $\mf{v} \in H^1_{\loc}(\Omega) \cap C(\Omega)$, $\mf{v} \not \equiv 0$, satisfy
\beq\label{coex}
\begin{cases}
\Delta v_i = M v_i \prod_{j \neq i} v_j^2 & \text{in $\Omega$} \\
v_i \ge 0 & \text{in $\Omega$}.
\end{cases}
\eeq

\begin{lemma}[Local Pohozaev identity]\label{lem: poh}
For every $x_0 \in \Omega$ and $r \in (0, \dist(x_0,\pa \Omega))$, we have that
\[
\begin{split}
r \int_{S_r(x_0)} \Big( \sum_{i=1}^3 |\nabla v_{i}|^2 + & M \prod_{j=1}^3 v_{j}^2 \Big)d\sigma = (N-2)\int_{B_r(x_0)} \sum_{i=1}^3 |\nabla v_{i}|^2\,dx \\
&+ N \int_{B_r(x_0)} M \prod_{j=1}^3 v_{j}^2\,dx + 2r \int_{S_r(x_0)} \sum_{i=1}^3 (\pa_\nu v_{i})^2\,d\sigma.
\end{split}
\]
\end{lemma}

\begin{proof}
We multiply the equation for $v_{i}$ by $\nabla v_{i}(x) \cdot (x-x_0)$, integrate over $B_r(x_0)$, and take the sum over $i$ from $1$ to $3$: we obtain
\[
\int_{B_r(x_0)} \Big( \sum_i \nabla v_{i} \cdot \nabla (\nabla v_{i} \cdot (x-x_0))  + M \sum_i v_{i} \nabla v_{i} \cdot (x-x_0) \prod_{j \neq i} v_{j}^2 \Big)dx
= r\int_{S_r(x_0)} \sum_i (\pa_\nu v_{i})^2\,d\sigma
\]
Now standard computations as in the classical Pohozaev identity give the result.
\end{proof}

\begin{proof}[Proof of Theorem \ref{thm: bas prop}]
Plainly, by minimality, each $\mf{v}_n$ solves system \eqref{coex} in $\Omega$ with $M=M_n \to +\infty$. The regularity of $\mf{v}$ follows then by Theorem \ref{thm: holder bounds}. Let $K \subset \Omega$ be a compact set. By ($iii$) and the local uniform convergence, we deduce that $v_1\,v_2\,v_3 \equiv 0$ in $\Omega$. Moreover, exactly as in Lemma \ref{lem: v bdd e M unb}, points ($ii$) and ($iii$) in Definition \ref{def: L} allow to prove that each $v_i$ is harmonic in its positivity set.

Let now $x_0 \in \Omega$ and $r \in (0,\dist(x_0,\pa \Omega))$. We observe that, by Lemma \ref{lem: poh}, for $r_0<r$
\[
\begin{split}
\int_{A_{r_0,r}(x_0)} \Big( \sum_{i=1}^3 |\nabla v_{i,n}|^2 + & M_n \prod_{j=1}^3 v_{j,n}^2 \Big)dx = \int_{r_0}^r \frac{N-2}{s}\int_{B_s(x_0)} \sum_{i=1}^3 |\nabla v_{i,n}|^2\,dx \, ds \\
&+  \int_{r_0}^r \frac{N}{s} \int_{B_s(x_0)} M_n \prod_{j=1}^3 v_{j,n}^2\,dx \, ds + 2 \int_{A_{r_0,r}(x_0)} \sum_{i=1}^3 \left(\nabla v_{i,n} \cdot \frac{x-x_0}{|x-x_0|}\right)^2\,dx,
\end{split}
\]
where $A_{r_0,r}(x_0) := B_r(x_0) \setminus \overline{B_{r_0}(x_0)}$. Recalling points ($ii$) and ($iii$) in the definition of $\mathcal{L}_{\loc}(\Omega)$, we can take the limit as $n \to \infty$, by dominated convergence, and deduce that
\[
\int_{A_{r_0,r}(x_0)}  \sum_{i=1}^3 |\nabla v_{i}|^2 \,dx  = \int_{r_0}^r \frac{N-2}{s}\int_{B_s(x_0)} \sum_{i=1}^3 |\nabla v_{i}|^2\,dx \, ds
+  2 \int_{A_{r_0,r}(x_0)} \sum_{i=1}^3 \left(\nabla v_{i} \cdot \frac{x-x_0}{|x-x_0|}\right)^2\,dx,
\]
for every $r_0<r$. By the fundamental theorem of calculus, this completes the proof.
\end{proof}

\end{document}